\documentclass[12pt,a4paper]{article}

\usepackage{mathrsfs}
\usepackage{amsfonts}
\usepackage{amssymb}
\usepackage{amsmath}
\newtheorem{theorem}{Theorem}[section]

\numberwithin{equation}{section}

\newtheorem{corollary}[theorem]{Corollary}

\newtheorem{definition}[theorem]{Definition}

\newtheorem{lemma}[theorem]{Lemma}

\newtheorem{remark}[theorem]{Remark}

\newenvironment{proof}[1][Proof]{\textbf{#1.}}{\ \rule{0.5em}{0.5em}}%

\begin{document}
\parindent 9mm
\title{A Class of Linear Boundary Systems
with Delays in State, Input and Boundary Output
\thanks{This work was supported by the Natural Science Foundation of China (grant nos. 11301412 and 11131006), Research Fund for the Doctoral Program of Higher Education of China (grant no. 20130201120053), Shaanxi Province Natural Science Foundation of China (grant no. 2014JQ1017), Project funded by China Postdoctoral Science Foundation (grant nos. 2014M550482 and 2015T81011), the Fundamental Research Funds for the Central Universities (grant no. 2012jdhz52). Part of this work was done during the first author was visiting at University of Wuppertal, Germany. The first author is grateful to Prof. Birgit Jacob for her warmful help both in his study and living.}
}
\author{ Zhan-Dong Mei
\thanks{Corresponding
author, School of Mathematics and Statistics, Xi'an Jiaotong
University,
 Xi'an
710049, China; Email: zhdmei@mail.xjtu.edu.cn } \ \ \
 Ji-Gen Peng \thanks{School of Mathematics and Statistics, Xi'an Jiaotong
University,
 Xi'an
710049, China; Email: jgpeng@mail.xjtu.edu.cn}}


\date{}
\maketitle \thispagestyle{empty}
\begin{abstract}
%
%
In this paper, we consider a class of linear boundary systems with
delays in state, input and boundary output. We prove the
well-posedness and derive some spectral properties of linear system
with delayed boundary feedback under some regularity conditions.
Moreover, we show the regularity of linear boundary systems with
delays in state and boundary output. With the above results, the
regularity of linear boundary systems with delays in state, input
and boundary output is verified. As applications, we prove the
well-posedness and the asymptotic behavior of population systems
with bounded and unbounded birth processes
``$B_1(t)=\int_0^\infty\int_{-r}^0\beta_1(\sigma,a)u(t-\tau,a)d\sigma
da$" and ``$B_2(t)=\int_0^\infty\beta_2(a)u(t-\tau,a)da$", and the well-posedness of population systems with death caused by harvesting.

\vspace{0.5cm} 

%
%

\end{abstract}


\section{Introduction}
Let $X,\ U,\ V,\ Y$ be Banach spaces. Denote by $L(X,Y)$ all the bounded linear operators from $X$ to $Y$. Then $L(X,Y)$ is a Banach space. Replace $L(X,X)$ with $L(X)$ for brief. Denote by $I$ the 
unit operator in $X$. Let $R$ be the set of all the real numbers and $R^+=\{s\in R:s\geq 0\}$. Assume that $A$ is a linear operator in $X$. Let $\rho(A)$, $\sigma(A)$ and $\sigma_P(A)$
be the resolvent set, spectrum and point spectrum of $A$, respectively. Denote by $R(\lambda,A)=(\lambda-A)^{-1}$ the resolvent operator of $A$. If $A$ generates a $C_0$-semigroup, then we denote by $T_A$ the corresponding semigroup (for the definition of $C_0$-semigroup, we refer to \cite{Engel2000}).
For $p\geq1$, $L^p((0,b); X)$
denotes the space of $X$-valued Bochner integrable functions
$u:(0,b)\rightarrow X$ with the norm
$\|u\|_{L^p((0,b);X)}=\big(\int_{0}^b\|u(t)\|dt\big)^{\frac{1}{p}}$.
For $\mathfrak{J}=(0,b)$, or $\mathfrak{J}=[0,b]$, the Sobolev spaces $W^{n,p}(\mathfrak{J}; X)$ is
defined by (\cite[Appendix]{Brezis1973}):
\begin{equation*}
W^{n,p}(\mathfrak{J}; X)=\{u|\ \exists \varphi\in L^{p}(\mathfrak{J};
X):u(t)=\sum_{k=0}^{n-1}c_{k}\frac{t^{k}}{k!}+\frac{t^{n-1}}{(n-1)!}\ast
\varphi(t),\ t\in \mathfrak{J}\}.
\end{equation*}
Obviously $W^{0,p}(\mathfrak{I}; X)=L^{p}(\mathfrak{I}; X)$. Let $W_{lock}^{n,p}(\mathfrak{I}; X)=\{f\in W^{n,p}(\mathfrak{I}; X):\mathfrak{I}\subset \mathfrak{J}$ is any bounded interval $\}$.

Consider linear boundary systems with delays in state, input and
boundary output described by
\begin{eqnarray}\label{bdso}
\left\{
  \begin{array}{ll}
        \dot{w}(t)=A_m w(t)+Lw_t, & \hbox{ }t\geq 0,\\
    Pw(t)=v(t),& \hbox{ }t\geq 0,\\
    y(t)=Mw(t)+Kw_t, & \hbox{ }t\geq 0,
  \end{array}
\right.
\end{eqnarray}
and boundary system with delays in state, input and output
\begin{eqnarray}\label{bdsio}
\left\{
  \begin{array}{ll}
        \dot{w}(t)=A_m w(t)+Lw_t+Eq_t, & \hbox{ }t\geq 0,\\
    Pw(t)=v(t),& \hbox{ }t\geq 0,\\
    y(t)=Mw(t)+Kw_t+Hq_t, & \hbox{ }t\geq 0,
  \end{array}
\right.
\end{eqnarray}
where $w$ take values in $X$ and $w_t$ is the history
function defined by $w_t(\theta)=w(t+\theta),\ \theta\in[-r,0]$; $q$
take values in $V$ and $q_t$ is the history function
defined by $q_t(\theta)=q(t+\theta),\ \theta\in[-r,0]$; $v$ and $y$
take values in $V$ and $Y$, respectively;
 $A_m$ is a bounded linear operator from $D(A_m)$
to $X$, $D(A_m)$ is a Banach space densely and continuously embedded
into $X$; $L\in L(W^{1,p}([-r,0],X),X)$; $E\in
L(W^{1,p}([-r,0],V),X)$;
 $P\in L(D(A_m),U)$ is a surjective; $M\in L(D(A_m),Y)$; $K\in L(W^{1,p}([-r,0],X),Y)$; $H\in
L(W^{1,p}([-r,0],X),Y)$.

 In real problems, because of physics and
technology, controllers and sensors are usually placed on the
boundaries of the systems. Although they are easy to be realized in
Physics, boundary control and observation bring many difficulties to
the study of infinite-dimensional linear system because they make
the control operator and observation operator unbounded. In 1983, Ho
and Russell \cite{Ho1983} discussed a class of boundary control
systems, whose state is not unbounded enough to escape from the
energy space when the initial state is in the energy space; they
call such control operator to be admissible. In 1987, by using the
Kalman's axiomatization method, Salamon \cite{Salamon1987}
established the theory of well-posed linear system whose state and
output are continuously depended on the initial state and input.
Later, Weiss \cite{Weiss1989a,Weiss1989b} simplified Salamon's
theory and call the control and observation operator to be
``admissible". In \cite{Weiss1989c}, Weiss defined and developed the
notion of regular linear systems, a subclass of well-posed linear
system. Well-posed and regular linear systems in the sense of
Salamon-Weiss is very important because many properties of them are
similar to that of finite-dimensional linear system; they became the
maximal theory frame of infinite-dimensional linear system in the
abstract sense over the past 30 years. There emerged many works on
the theory of admissibility and regular linear systems. The
well-posedness and/or regularities of many physical systems such as
wave systems, Schr$\ddot{o}$dinger equation, beam and Naghdi system
\cite{Chai2010,Guo2005a,Guo2005,Guo2012}, have been proved.

Delays are usually inevitable to appear in state, input and/or
output. The existence of delays produces many difficulties to
analyze the well-posedness and regularity of systems because it even
makes finite-dimensional system infinite-dimensional. The delayed
freedom systems (without input) have been studied for many years.
Hale \cite{Hale1971} and Webb \cite{Webb1974} were among the first
who applied semigroup methods to the study of such equations; but
the state spaces are of finite dimension. For specifical
infinite-dimensional systems, such as wave and beam equations, many
authors convert the delay equations with $Lw_t=kw(t-r)$ to undelayed
equations by introduce a new variable $z(t,\tau):=w(t-\tau r)$. In
such a way, there hold $\frac{\partial z(t,\tau)}{\partial
t}=-\frac{1}{r}\frac{\partial z(t,\tau)}{\partial \tau}$ and $Lw_t=kz(t,1).$ Then
the delayed part disappears by increasing a new equation, see
\cite{Ammari2010,Guo2010,Pignotti2012,Shang2012}. For the systems with distribute delays $Lw_t=\int_{-r}^0d\mu(\sigma)w(t+\sigma),$ by introducing variable $z(t,\tau,s):=w(t-\tau s)$, $s\in[0,r]$, one can obtain $\frac{\partial z(t,\tau,s)}{\partial
t}=-\frac{1}{s}\frac{\partial z(t,\tau,s)}{\partial \tau}$ and $Lw_t=\int_{-r}^0d\mu(\sigma)z(t,1,-\sigma)$.
Then the delayed systems are also transferred to undelayed systems \cite{Nicaise2008}. The well-posedness of the systems were studied by using Hilbert space method and the corresponding system operators are dissipative. In order to study general delayed linear system with infinite dimensional state spaces,
B$\acute{a}$tkai et. \cite{Batkai2001} introduced the perturbation theory of semigroups. Concretely, they transferred the
delayed freedom system to a larger undelayed system and use perturbation theory to prove that the system operator generates a $C_0$-semigroup and use spectral theorem to study the asymptotic behavior. The theory of well-posed linear system can also be used the study a class of general delayed linear systems \cite{Salamon1987}. In the series of their papers, Hadd et al. studied the mild expressions and regularities of general delayed linear system \cite{Hadd2005,Hadd2005b,Hadd2006,Hadd2006b}.
Observe that the controller of the systems Hadd et al. studied are placed
on the interior. However, like undelayed system, in the real problem the controller are usually
placed on the boundary. Therefore, it is urge to develop a theory to
solve the well-posedness and regularity of general delayed linear
system with boundary control and boundary observation.

Population dynamical systems with delay birth process can be
described as system $(\ref{bdso})$ with $L=0,\ M=0$ and $v(t)=y(t)$. In \cite{Piazzera2004}, Piazzera considered the situation that the
birth process, namely, the boundary feedback operator $K$, is a
bounded linear operator with respect to the history function.
Concretely, he proved the well-posedness of such system by
Desch-Schappacher perturbation theorem \cite{Desch1989} and
discussed the asymptotic behavior through spectral theory and
positive semigroup theory. Simultaneously, Piazzera
 pointed out that population dynamical system
with unbounded birth process
``$B(t)=\int_0^\infty\beta(a)u(t-\tau,a)da, t\geq 0$" is an {\it
open problem}, and only particular results, e.g. for neutral
differential equations \cite{Nagel2003} or analytic semigroups
\cite{Greiner1991}, are known while a general perturbation result is
still missing. The main difficulty of population dynamical systems
with unbounded birth process lies in that the unboundedness makes
Desch-Schappacher perturbation theorem invalid. In the recent paper
\cite{Mei2015}, we solved such {\it open problem} by using feedback theory
of the regular linear system developed by Weiss \cite{Weiss1994b}.
In \cite{Mei2015a}, we proved the well-posedness of system with $L$ being unbounded and $K$ being bounded by our admissible invariable theorem developed in \cite{Mei2010}.
To the best of the authors' knowledge,
there has no work that proved the well-posedness of the system with
$K$ and $L$ being unbounded.
The asymptotic behaviors of population systems with $L\neq 0$ have not been studied yet. Furthermore, the well-posedness of population systems with death caused by harvesting ($E\neq0$) also have not been solved. Motivated by this, we will try to use
the theory of regular linear system to deal with such problem.
However, we observe that the unboundedness of $L$ will bring us
essential difficulties. In order to settle such problem, we plan to
use the perturbation theory developed by our recent paper
\cite{Mei2010a}. Moreover, some other theorems will be proved, which
are useful to deal with our problem.

The rest of this paper is arranged as follows. Section 2 will
recall the theory of regular linear system, which is the main
tool in our paper. By means of the theory of regular linear system,
we get in Section 3 the well-posedness and spectrum relations of
linear system with delayed boundary feedback under some regularity
conditions. In Section 4, the regularity of linear boundary systems
with delays in state and boundary output is proved. With the results obtained in Sec. 4, we derive the regularity of linear boundary systems with delays in state, input and boundary output.
Moreover, we prove such bounded feedback systems are abstract linear control systems. As applications, we firstly study the well-posedness and asymptotic behavior of population dynamical system with death caused by pregnancy and with delayed birth process, secondly prove
the well-posedness of population systems with death caused by harvesting.

\section{Preliminaries on Regular Linear Systems}

In this section, we shall recall the theory of well-posed linear
system in the sense of Salamon-Weiss \cite{Salamon1987} and regular
linear system in the sense of Weiss \cite{Weiss1994a}. Throughout
this section, we assume that $X$, $U$ and $Y$ are Banach spaces,
$1<p<\infty$. Let $T=\{T(t)\}_{t\geq 0}$ be a $C_0$-semigroup and
$A$ its generator on $X$. Denote by $X_{-1}$ the extrapolation space
corresponding to $X$, which is the completion of $X$ under the norm
$\|R(\lambda_0,A)\cdot\|$ with $R(\lambda_0,A)$ the resolvent of $A$
at $\lambda_0$; $\{T_{-1}(t)\}_{t\geq0}$ is the extrapolation
semigroup of $\{T(t)\}_{t\geq0}$ with generator $A_{-1}$, which is
the continuous extension of $\{T(t)\}_{t\geq0}$ on $X_{-1}$. For
more details of extrapolation space and extrapolation semigroup, we
refer to \cite{Engel2000}.

The pair $(T,\Phi)$ is called {\it abstract linear control system},
if $\Phi=\{\Phi(t)\}_{t\geq 0}$ is a family of bounded linear
operators from $L^p(R^+,U)$ to $X$ such that
\begin{eqnarray*}
 \Phi(t+\tau)u=T(t)\Phi(\tau)u+\Phi(t)u(\tau+\cdot),\ u\in
L^p(R^+,U).
\end{eqnarray*}
It follows by \cite{Weiss1989a} that there exists a unique operator $B\in
L(U,X_{-1})$, called admissible control operator, such that
$$\Phi(t)u=\int_0^tT_{-1}(t-s)Bu(s)ds.$$
In this case, we say $(T,\Phi)$ is generated by $(A,B)$ and denote
$\Phi=\Phi_{A,B}$.

The pair $(T,\Psi)$ is called {\it abstract linear observation
system}, if $\Psi=\{\Psi(t)\}_{t\geq 0}$ is a family of bounded
linear operators from $X$ to $\L^p(R^+,Y)$ such that
\begin{eqnarray}
  (\Psi(t+\tau) x)(\cdot)=(\Psi(t)T(\tau)x)(\cdot-\tau)
\mbox{ on } [\tau,t+\tau],\ x\in X.
\end{eqnarray}
By \cite{Weiss1989b}, it follows that there exists a unique operator $C$, called
admissible observation operator, such that
$$(\Psi(t)x)(\sigma)=CT(\sigma)x,\
\sigma\in [0,t].$$ In this case, we say $(T,\Psi)$ is generated by
$(A,C)$ and denote $\Psi=\Psi_{A,C}$. By \cite{Weiss1989b}, there
exists a unique operator $\Psi(\infty):\hbox{ }X\rightarrow
L^2_{loc}(R^+,Y)$ such that
\begin{eqnarray*}
  \Psi(\tau)=P_\tau\Psi(\infty),\ \tau\geq 0.
 \end{eqnarray*}

The pair $(T,\Phi,\Psi,F)$ is called {\it well-posed linear system},
if $(T,\Phi)$ is abstract linear control system, $(T,\Psi)$ is abstract linear observation system, and $F=\{F(t)\}_{t\geq 0}$ is a family of bounded linear operators
from $L^p(R^+,U)$ to $L^p(R^+,Y)$ such that
\begin{eqnarray}\label{F}
  (F(t+\tau)u)(\cdot)=(\Psi(t)\Phi(\tau)u+F(t)u(\tau+\cdot))(\cdot-\tau)
\mbox{ on } [\tau,t+\tau],\ u\in L^p(R^+,U).
\end{eqnarray}
It follows from \cite{Weiss1989b} that there exists a unique
operator $F(\infty):\hbox{ }L^p_{loc}(R^+,U)\rightarrow
L^p_{loc}(R^+,Y)$ such that
\begin{eqnarray*}
  F(\tau)(t)=F(\infty)(t),\ 0\leq t\leq \tau.
\end{eqnarray*}

The well-posed linear system $\Sigma$ is called to be {\it regular},
if the limit
\begin{equation}\label{rD}
    \lim\limits_{t\rightarrow 0}\frac{1}{t}\int_0^t(F_\infty
u_0)(s)ds
\end{equation}
exist, where $u_0(t)=z,z\in U,t\geq 0$. The operator $D\in L(U,Y)$
defined by $Dz=\lim\limits_{t\rightarrow
0}\frac{1}{t}\int_0^t(F_\infty u_0)(s)ds$ is the feedthrough
operator. In this case, we also say that $\Sigma=(T,\Phi,\Psi,F)$ is
generated by $(A,B,C,D)$, and we denote $F=F_{A,B,C,D}$. Moreover,
we denote $F_{A,B,C}$ by $F_{A,B,C,0}$ for brief.

In \cite{Weiss1994a}, Weiss introduced an extension of $C$,
called {\it $\Lambda$-extension} with respect to $A$, which is
defined by
\begin{equation}
C^A_\Lambda x=\lim_{\lambda\rightarrow \infty}C\lambda R(\lambda,A)x,\ D(C^A_\Lambda)=\{x\in X:\mbox{this above limit exists
in}\, Y\}
\end{equation}
It follows from \cite[Theorem 4.5 and Proposition
4.7]{Weiss1989b} that for any $x\in X$, $y(t)=C_\Lambda T(t)x$ a.e.
in $t\geq 0$ whenever $C$ is admissible for $A$.

The transfer function $G$ of regular linear system generated by
$(A,B,C,D)$ is given by
\begin{eqnarray}
  G(\lambda)=C^A_\Lambda R(\lambda,A_{-1})B+D, \hbox{
  }Re(\lambda)>w_0(T),
\end{eqnarray}
where $w_0(T)$ is the growth bound of the semigroup $T$, and we
denote $G=G_{A,B,C}$.

In order to state the following theorem, we define
\begin{align*}
  D^p(M)
  =\{f(\cdot)\in L^p_{loc}(R^+,X):
  f\in D(M) \hbox{ }for \hbox{ a.e. }t\geq 0, and \hbox{ }
  Mf(\cdot)\in L^p_{loc}(R^+,X)\}.
\end{align*}

\begin{theorem}\cite{Weiss1989c}\label{gm}
Let $\Sigma$ be a regular linear system with generating operator
$A$, $B$, $C$ and $D$ on $(X,U,Y)$. Then, for given $(x_0,u)\in
X\times L^p(R^+,U)$, the state trajectory $x(\cdot)$ of $\Sigma$,
given by $x(t):=T(t)x_0+\Phi_{A,B}u$, $t\geq 0$, is a.e.
differential in $X_{-1}$ and
\begin{eqnarray}
  \dot{x}(t)=A_{-1}x(t)+Bu(t), \hbox{ }x(0)=x_0 \hbox{ } for \hbox{
  a.e. } t\geq 0.
\end{eqnarray}
Furthermore, $x(t)\in D(C^A_\Lambda)$ for a.e. $t\geq 0$ and the
output function $y=\Psi_{A,C}(\infty)x+F_{A,B,C,D}(\infty)u$ of
$\Sigma$ is given by
\begin{eqnarray}\label{otp}
  y(t)=C_\Lambda^Ax(t)+Du(t) \hbox{ } for \hbox{ } a.e. \hbox{ }t\geq
  0.
\end{eqnarray}
In particular, $\Phi_{A,B}(\cdot)u\in D^p(C^A_\Lambda)$ and
the extended input-output map $F(\infty)$ is given by
\begin{eqnarray}\label{Fu}
  (F(\infty)u)(t)=C^A_\Lambda\int_0^tT_{-1}(t-s)Bu(s)ds+Du(t) \hbox{ }
  for \hbox{ }a.e. \hbox{ }t\geq0.
\end{eqnarray}
\end{theorem}

\begin{definition}\cite{Staffans2005,Weiss1994b}
An operator $\Gamma\in L(Y,U)$ is called an admissible feedback operator
for $\Sigma=(T,\Phi,\Psi,F)$ if $I-F(t)\Gamma$ is invertible for any $t\geq 0$ (hence any $t\geq 0$).
\end{definition}

\begin{theorem}\label{relation}\cite{Weiss1994b}
Let $(A,B,C,D)$ be the generator of regular linear system
$\Sigma=(T,\Phi,\Psi,F)$ on $(X,U,Y)$ with admissible feedback
operator $\Gamma\in L(Y,U)$. Then $(I-D\Gamma)_{left}^{-1}$ exists and the feedback system $\Sigma^\Gamma$ is
a well-posed linear system generated by
$(A^\Gamma,B^\Gamma,C^\Gamma)$:
\begin{eqnarray*}
  &A^\Gamma =(A_{-1}+B\Gamma(I-D\Gamma)_{left}^{-1} C^A_\Lambda)|_X,\;\\
 &D(A^\Gamma): = \{z\in D(C^A_\Lambda):(A_{-1}+B\Gamma (I-D\Gamma)_{left}^{-1}C^A_\Lambda)z\in
 X\}
\end{eqnarray*}
and $C^\Gamma=(I-D\Gamma)_{left}^{-1}C^A_\Lambda$ restricted to
$D(A^\Gamma)$, where $J^{A,A^\Gamma}$ is defined by
$J^{A,A^\Gamma}x=\lim_{\lambda\rightarrow
\infty}(\lambda-A_{-1})^{-1}x$ (in $X_{-1}^{A^\Gamma}$) with
$D(J^{A,A^\Gamma})=\{x\in X^A_{-1}:$ the limit
$\lim_{\lambda\rightarrow \infty}(\lambda-A_{-1})^{-1}x$ exists
$\}$. If in addition, $I-D\Gamma$ is invertible, we have that
 $B^\Gamma=J^{A,A^\Gamma}B$ and $D(C^A_\Lambda)=D((C^\Gamma)^{A^\Gamma}_\Lambda)$.
\end{theorem}

The following theorem is an important tool in this paper, which was
proved in our paper \cite{Mei2010a}.

\begin{lemma}\cite{Mei2010a}\label{observe}
Assume that $(A,B,C)$ and $(A,B,P)$ generate regular linear systems
on $(X,U,Y)$ and $(X,U,X)$, respectively. Then $(A+P,J^{A,A+P}B,C)$
generates a regular linear system. Moreover, there hold
\begin{align}\label{Frelation1}
\Phi_{A+P,J^{A,A+P}B} =\Phi_{A+P,I}F_{A,B,P} +\Phi_{A,B}.
\end{align}
and
\begin{align}\label{Frelation}
F_{A+P,J^{A,A+P}B,C} =F_{A+P,I,C}F_{A,B,P} +F_{A,B,C}.
\end{align}

\end{lemma}

In (\ref{Frelation}), we replace $B$ with $I$ to get $F_{A+P,I,C}
=F_{A+P,I,C}F_{A,I,P} +F_{A,I,C}$. Hence
\begin{align}\label{teshu}
F_{A+P,I,C}= F_{A,I,C}(I-F_{A,I,P})^{-1}.
\end{align}
Substitute (\ref{teshu}) into (\ref{Frelation}) to derive
\begin{align}\label{main}
    F_{A+P,J^{A,A+P}B,C}
=F_{A,I,C}(I-F_{A,I,P})^{-1}F_{A,B,P} +F_{A,B,C}.
\end{align}
Similarly, we can obtain that
\begin{align}\label{phiab}
    \Phi_{A+P,J^{A,A+P}B}
=\Phi_{A,I}(I-F_{A,I,P})^{-1}F_{A,B,P} +\Phi_{A,B}.
\end{align}

Obviously, the following statement hold.\\
\textbf{(S1)} System $(A,B,C)$ is a regular linear system if and only if $(A,B,C,D)$ is a regular linear system.
    In this case, there holds
   $$(F_{A,B,C,D}(s)u)(t)=(F_{A,B,C}(s)u)(t)+Du(t), \ u\in L^p(R^+,U),\ s\geq t\geq 0.$$
\textbf{(S2)} $(A,B,C)$ generating a regular linear system
      $(T,\Phi,\Psi,F)$ indicates that $(A,BW,VC)$ generates a regular linear system $(T,\Phi W,V\Psi,VFW)$, where $W$ and $V$ are bounded linear operators.
    In this case, there holds
     $$F_{A,BW,VC}=VF_{A,B,C}W.$$
\textbf{(S3)} If $(A,B)$ is an abstract linear control system
and $P$ is a bounded linear operator, then $(A,B,P)$ is a regular linear system with admissible feedback operator $I$. Moreover,
$\|F_{A,B,P}(t)\|\rightarrow 0$ as $t\rightarrow 0^+$, (the proof is similar to \cite[Lemma 3.2]{Hadd2005a}).\\
\textbf{(S4)} If $(A,C)$ is an abstract linear observation system
and $M$ is a bounded linear operator, then $(A,M,C)$ is a  regular linear system with admissible feedback operator $I$. Moreover,
$\|F_{A,M,C}(t)\|\rightarrow 0$ as $t\rightarrow 0^+$, (the proof is similar to \cite[Lemma 3.8]{Mei2010}).

Linear boundary system is described in the abstract frame as follows
\cite{Malinen2006,Salamon1987}.
 \begin{eqnarray}\label{boundctr}
 \left\{
   \begin{array}{ll}
     \dot{z}(t)&=Lz(t), \\
     Gz(t)&=u(t), \\
     y(t)&=Kz(t),
   \end{array}
 \right.
 \end{eqnarray}
 where $L$, $G$ and $K$ are linear operators on $D(L)$, $D(L)$
          is continuously embedded in $X$;
$\bigg[\begin{array}{c}
            L \\
            G \\
            K
          \end{array}\bigg]
   $ is closed operator from $D(L)$ to space $X \times U\times Y$;
  $G$ is surjection and $Ker\{G\}:=\{z\in Z: Gz=0\}$ is dense in $X$;
  $L|_{Ker\{G\}}$ generates a $C_0$-semigroup on $X$.
We denote system (\ref{boundctr}) by $(L,G,K)$ for brief.

Denote $A=L|_G$, $C=K|_{D(A)}$. By \cite{Greiner1987}, $D(L)$ can be
decomposed to direct sum $D(L)=D(A)\bigoplus Ker\{\lambda-L\}$ and
the operator $G$ is bijective from $Ker\{\lambda-L\}$ onto $U$,
where $\lambda$ is any component of resolvent set of $A$. Hence we
can denote $D_{\lambda,L,G}$ by the solution operator from $z$ to
$u$ of the following function
\begin{align*}
    \left\{
      \begin{array}{ll}
        (\lambda-L)z=0, & \hbox{} \\
        Gz=u, & \hbox{}
      \end{array}
    \right.
\end{align*}
that is $z=D_{\lambda,L,G}u.$ We can define $B=(\lambda-A_{-1})D_{\lambda,L,G}\in L(U,X_{-1})$.
Then there holds
\begin{align}\label{ZB}
    D_{\lambda,L,G}=(\lambda-A_{-1})^{-1}B,\
D(L)=D(A)\bigoplus (\lambda-A_{-1})^{-1}BU,
\end{align}
and 
\begin{align}\label{GB}
    G(\lambda-A_{-1})^{-1}B=I.
\end{align}

By \cite{Malinen2006,Salamon1987,Staffans2005}, it follows that, for any $z(0)\in X$, $u\in W^{2,p}(R^+,U)$
satisfying $A_{-1}z(0)+Bu(0)\in X$, we have $z(\cdot)\in
C(R^+,D(L))$ and $y(\cdot)\in C(R^+,Y)$. There holds
$z(t)=Ax(t)+Bu(t)$ and
$y(t)=C(x(t)-(\lambda-A_{-1})^{-1}Bu(t))+K(\lambda-A_{-1})^{-1}Bu(t)$.
Then $B$ and $C$ are the control and observation operator of
$(L,G,Q)$, respectively. Boundary system $(L,G,K)$ is well-posed if
there exist positive function $m$ and $n$ on $R^+$ such that
$$\|x(t)\|+\|y\|_{L^{p}([0,t],Y)}\leq m(t)\|x(0)\|+n(t)\|u\|_{L^([0,t],U)},\ t\geq 0.$$
The corresponding transform function is $KD_\lambda$. It is regular
if it is well-posed and the limit $\lim_{\lambda\rightarrow
+\infty}KD_\lambda u$ exist for any $u\in U$. In this case, we
denote $\bar{K}u:=\lim_{\lambda\rightarrow +\infty}KD_\lambda u$ and
$\bar{K}$ is called to be the feedthrough operator. Then boundary
system $(L,G,K)$ is regular with generator $(A,B,C,\bar{K})$. We also say that the generator is $(A,B,K,\bar{K})$ and denote $K_\Lambda^A$ by $C_\Lambda^A=(K|_{D(A)})_\Lambda^A$.

The following observation is obviously: \\
\textbf{(S5)} If $(L,G)$ is an abstract
linear control system and $K$ is bounded from $X$ to $Y$, then
$(L,G,K)$ is a regular linear system generated by $(A,B,K)$ with admissible feedback operator $I$.

\section{Linear Boundary System with Delayed Boundary Feedback}
In order to deal with the delayed linear systems with boundary
control and observation, our first task is to transfer them to
linear system without delays. To do this, we observe that $w_t$ is a
solution of the boundary control system \cite{Curtain1996,Staffans2005}
\begin{align*}
    \left\{
      \begin{array}{ll}
        \frac{\partial}{\partial t}x(t,\theta)=\frac{\partial}{\partial \theta}x(t,\theta),\ x\in [-r,0]  & \hbox{ } \\
        x(t,0)=w(t). & \hbox{ }
      \end{array}
    \right.
\end{align*}

Denote $\mathfrak{X}=L^p([-r,0],X)),\
\mathfrak{A}_m:=\frac{d}{d\theta}$ with domain $D(\mathfrak{A}_m):=
W^{1,p}([-r,0],X);\ \mathfrak{G}F=F(0), \forall F\in L^p([-r,0],X)$.
Then $\mathfrak{A}:=\mathfrak{A}_m$ with domain $D(\mathfrak{A}) :=
Ker\{\mathfrak{G}\}$ is the generator of left shift $C_0$-semigroup
on $L^p([-r,0],X)$ and system is an abstract linear control system
generated by $(\mathfrak{A},\mathfrak{B})$ with
$\mathfrak{B}=(\lambda-\mathfrak{A}_{-1})e_\lambda$. Here
$e_\lambda$ is defined by $(e_\lambda x)(\theta)=e^{\lambda
\theta}x,\ x\in X, \ \theta\in[-r,0]$. System with $v(t)=y(t)$ is
described by a larger undelayed system
\begin{eqnarray*}
(ABS)\left\{
  \begin{array}{ll}
    \dot{x}(t)=\mathfrak{A}_mx(t),& \hbox{ } t\geq 0,\\
    \mathfrak{G}x(t)=w(t),& \hbox{ } t\geq 0,\\
    \dot{w}(t)=A_m w(t)+Lx(t), & \hbox{ }t\geq 0,\\
    Pw(t)=Mw(t)+Kx(t), & \hbox{ }t\geq 0,
  \end{array}
\right.
\end{eqnarray*}
which can be converted to the following boundary system
\begin{eqnarray*}
  \left\{
    \begin{array}{ll}
      \frac{d}{dt}\left(
                    \begin{array}{c}
                      x(t) \\
                      w(t) \\
                    \end{array}
                  \right)=
     \left(
                            \begin{array}{cc}
                              \mathfrak{A}_m & 0 \\
                              L & A_m \\
                            \end{array}
                          \right)
        \left(
                    \begin{array}{c}
                      x(t) \\
                      w(t) \\
                    \end{array}
                  \right)
      , & \hbox{ } t\geq 0\\
      \left(
                            \begin{array}{cc}
                              \mathfrak{G} & 0 \\
                              0 & P \\
                            \end{array}
                          \right)
      \left(
                    \begin{array}{c}
                      x(t) \\
                      w(t) \\
                    \end{array}
                  \right)
      =\left(
         \begin{array}{cc}
           0 & I \\
           K & M \\
         \end{array}
       \right)
        \left(
                    \begin{array}{c}
                      x(t) \\
                      w(t) \\
                    \end{array}
                  \right), & \hbox{ } t\geq 0.
    \end{array}
  \right.
\end{eqnarray*}

We denote $\mathcal {A}_{L,K,M}=\left(
                            \begin{array}{cc}
                              \mathfrak{A}_m & 0 \\
                              L & A_m \\
                            \end{array}
                          \right)$ with domain
$D(\mathcal {A}_{L,K,M})=\bigg\{\left(
                                \begin{array}{c}
                                  z \\
                                  f \\
                                \end{array}
                              \right)\in
                              D(\mathfrak{A}_m)\times D(A_m):
z(0)=f,Pf=Mf+Kz \bigg\}$.

\begin{definition}
We say function $w:[-r,+\infty)\rightarrow X$ is a classical solution of system (\ref{bdso}) with $v(t)=y(t)$, if $w\in C([-r,+\infty),X)\cap C^1([0,\infty),X)$, and there hold $w(t)\in D(A_m)$, $w_t\in W^{1,p}([-r,0],X)$, $\dot{w}(t)=A_mw(t)+Lw_t$,
and $Pw(t)=Mw(t)+Kx(t), \ t\geq 0$.
\end{definition}

Similar to the proof of \cite{Mei2015}, we can obtain the following
theorem.

\begin{theorem}\label{wel}
Assume that $\mathcal {A}_{L,K,M}$ generates a
$C_0$-semigroup on the space $X\times \partial X$. Then system (\ref{bdso}) with $v(t)=y(t)$ is well-posed, that is, for each
initial value there is a unique classical solution
 and it depends continuously on the initial data.
\end{theorem}

The rest tasks in this section is to prove the well-posed of system
(\ref{bdso}) with $v(t)=y(t)$ and derive the spectrum relations.
We firstly introduce two lemmas.

\begin{lemma}\label{output}
Assume that the boundary control system
\begin{align*}
  (BCS)\left\{
   \begin{array}{ll}
   \dot{z}(t)&=Lz(t)\\
     Gz(t)&=u(t)
   \end{array}
 \right.
\end{align*}
is an abstract linear control system generated by
$(\mathbb{A},\mathbb{B})$. Then the boundary system $(L,G,Q)$ is a
regular linear system on $(X,U,Y)$ if and only if
$(\mathbb{A},\mathbb{B},Q)$ generates a regular linear system. In
this case, for any $z\in Z$, we have
$$Qz=Q_\Lambda^\mathbb{A} z+\bar{Q} Gz.$$
\end{lemma}
\begin{proof}\ \
Obviously, system $(L,G,Q)$ generating a regular linear system implies
that $(\mathbb{A},\mathbb{B},Q)$ generating a regular linear system.

Next we shall prove the sufficiency. Assume that
$(\mathbb{A},\mathbb{B},Q)$ generates a regular linear system and the observation operator is $\mathbb{C}$. For
any $z\in D(L)$, we have $Gz\in U$. Assume that $\lambda\neq
\lambda_0$. The transform function of system $(L,G,Q)$ is given by
$Q(\lambda-\mathbb{A}_{-1})^{-1}\mathbb{B}$. By resolvent
equalities, it follows that
\begin{align*}
    (\lambda-\mathbb{A})^{-1}(\lambda_0-\mathbb{A}_{-1})^{-1}\mathbb{B}
=\frac{(\lambda_0-\mathbb{A}_{-1})^{-1}\mathbb{B}-(\lambda-\mathbb{A}_{-1})^{-1}\mathbb{B}}{\lambda-\lambda_0}.
\end{align*}
The assumption that $(\mathbb{A},\mathbb{B},Q)$ generates a regular
linear system implies
$(\lambda_0-\mathbb{A}_{-1})^{-1}\mathbb{B}Gz\in D(L)\subset
D(Q_\Lambda^\mathbb{A} )$. So
\begin{align}\label{feedth}
 \nonumber Q_\Lambda^\mathbb{A}(\lambda_0-\mathbb{A}_{-1})^{-1}\mathbb{B}Gz=& \lim_{\lambda\rightarrow +\infty}Q\lambda(\lambda-\mathbb{A})^{-1}(\lambda_0-\mathbb{A}_{-1})^{-1}\mathbb{B}Gz
\\
\nonumber=&\lim_{\lambda\rightarrow +\infty}Q\lambda\frac{(\lambda_0-\mathbb{A}_{-1})^{-1}\mathbb{B}Gz-(\lambda-\mathbb{A}_{-1})^{-1}\mathbb{B}Gz}{\lambda-\lambda_0}\\
=&Q(\lambda_0-\mathbb{A}_{-1})^{-1}\mathbb{B}Gz-\lim_{\lambda\rightarrow +\infty}Q(\lambda-\mathbb{A}_{-1})^{-1}\mathbb{B}Gz.
\end{align}
Since $G$ is surjective, we have that for any $u\in U$, limit
$\lim_{\lambda\rightarrow
+\infty}Q(\lambda-\mathbb{A}_{-1})^{-1}\mathbb{B}u$ exists. Set
$$\mathbb{D}u=:\lim_{\lambda\rightarrow
+\infty}Q(\lambda-\mathbb{A}_{-1})^{-1}\mathbb{B}u,\ u\in U.$$ Then
we rewrite (\ref{feedth}) by
\begin{align}\label{feedth1}
Q_\Lambda^\mathbb{A}(\lambda_0-\mathbb{A}_{-1})^{-1}\mathbb{B}Gz
=Q(\lambda_0-\mathbb{A}_{-1})^{-1}\mathbb{B}Gz-\mathbb{D}Gz.
\end{align}
By (\ref{ZB}), it follows that $z-(\lambda_0-\mathbb{A}_{-1})^{-1}\mathbb{B}Gz\in D(\mathbb{A})$. This implies that
\begin{align}\label{final}
    Q_\Lambda^\mathbb{A}(z-(\lambda_0-\mathbb{A}_{-1})^{-1}\mathbb{B}Gz)=
Q(z-(\lambda_0-\mathbb{A}_{-1})^{-1}\mathbb{B}Gz).
\end{align}
Combine (\ref{feedth1}) and (\ref{final}) to get
\begin{align}\label{ex}
 Q_\Lambda^\mathbb{A}z
=Qz-\mathbb{D}Gz.
\end{align}

For any $z(0)\in D(L)$ and $u\in W^{2,p}_{loc}(R^+,U)$ satisfy
$\mathbb{A}_{-1} z(0)+\mathbb{B}_{-1}u(0)\in X $,
we have $z(t)=T_{\mathbb{A}}(t)z(0)+\Phi_{\mathbb{A},\mathbb{B}}(t)u\in Z$ and
\begin{align}\label{11}
    \nonumber y(t)=& Qz(t)\\
\nonumber=& \mathbb{C}^\mathbb{A}_\Lambda z(t)+\mathbb{D}Gz(t)\\
\nonumber=& \mathbb{C}^\mathbb{A}_\Lambda
(T_{\mathbb{A}}(t)z(0)+\Phi_{\mathbb{A},\mathbb{B}}(t)u)+
\mathbb{D}u(t)\\
\nonumber=& [\Psi^\infty_{\mathbb{A},Q}z(0)](t)+
[F^\infty_{\mathbb{A},\mathbb{B},Q}u](t)+\mathbb{D}u(t)\\
=& [\Psi_{\mathbb{A},\mathbb{C}}(\infty)z(0)](t)+
[F_{\mathbb{A},\mathbb{B},\mathbb{C},\mathbb{D}}(\infty)u](t).
\end{align}
where the relation (\ref{ex}) and Theorem \ref{gm} is
used. Hence $(L,G,Q)$ is a regular linear system with generator
$(\mathbb{A},\mathbb{B},Q,\mathbb{D})$. By definition,
$\bar{Q}=\mathbb{D}$. The proof is therefore completed.
\end{proof}

\begin{lemma}\label{boundarycontrol}
Assume that the boundary system $(L,G,Q)$ is a regular linear system
generated by $(\mathbb{A},\mathbb{B},Q,\bar{Q})$ on $(X,U,X)$. Then,
the perturbed boundary system
\begin{eqnarray*}
 (CS)\left\{
   \begin{array}{ll}
   \dot{z}(t)&=Lz(t)+Qz(t)\\
     Gz(t)&=u(t)
   \end{array}
 \right.
 \end{eqnarray*}
is an abstract linear control system generated by
$(\mathbb{A}+Q,J^{\mathbb{A},\mathbb{A}+Q}\mathbb{B}+\bar{Q})$.
Moreover, the state $z(\cdot)$ of the equation satisfies
$z(\cdot)\in D^p(Q_\Lambda^\mathbb{A})$ and it can be expressed by
\begin{align}\label{express}
    z(t)=T_\mathbb{A}(t)x+\int_0^tT_\mathbb{A}(t-s)
[Q_\Lambda^\mathbb{A}z(s)+\bar{Q}u(s)]ds+\Phi_{\mathbb{A},\mathbb{B}}(t)u.
\end{align}

\end{lemma}
\begin{proof}\ \
 By the proof of Theorem 3.9 in \cite{Mei2010},
it follows that $\bigg(\mathbb{A},\left(
 \begin{array}{cc}
 I & \mathbb{B} \\
 \end{array}
  \right), \left(
  \begin{array}{c}
  Q \\
  0 \\
   \end{array}
   \right)\bigg)
$ generates a regular linear system, and $I$ is an admissible
feedback operator. Then, the system operator, control operator and
observation operator of the closed loop system are given by:
$\mathbb{A}^I=\mathbb{A}+Q$,  $\left(
 \begin{array}{cc}
 I & \mathbb{B} \\
 \end{array}
  \right)^I=\left(
 \begin{array}{cc}
 I & J^{\mathbb{A},\mathbb{A}+Q}\mathbb{B} \\
 \end{array}
  \right)$ and
$\left(
  \begin{array}{c}
  Q \\
  0 \\
   \end{array}
   \right)^I=\left(
  \begin{array}{c}
  Q \\
  0 \\
   \end{array}
   \right),$ respectively.

Since $Q$ is admissible for $\mathbb{A}$, we derive that
$\mathbb{A}+Q$ generates a $C_0$-semigroup on $X$. Denote by
$\mathbb{B}_0$ the control operator corresponding to system $(CS)$.
Then, $$\mathbb{B}_0Gz=Lz+Qz-(\mathbb{A}+Q)_{-1}z, \forall z\in D(L).$$

Next we shall prove that
$\mathbb{B}_0=J^{\mathbb{A},\mathbb{A}+Q}\mathbb{B}.$ By
\cite[(7.14)]{Weiss1994b} and Theorem \ref{relation}, we obtain that
for any $\forall \mbox{ } x\in(\beta-A_{-1})D\bigg(\left(
  \begin{array}{c}
  Q \\
  0 \\
   \end{array}
   \right)^\mathbb{A}_\Lambda\bigg)=(\beta-A_{-1})D(Q^\mathbb{A}_\Lambda)$,
\begin{eqnarray}\label{J}
\nonumber J^{\mathbb{A},\mathbb{A}+Q}x&=&(\beta-
(\mathbb{A}+Q)_{-1})(\beta-\mathbb{A}_{-1})^{-1}x+\left(
 \begin{array}{cc}
 I & \mathbb{B} \\
 \end{array}
  \right)^I\left(
  \begin{array}{c}
  Q \\
  0 \\
   \end{array}
   \right)^\mathbb{A}_\Lambda(\beta-\mathbb{A}_{-1})^{-1}x\\
&=&(\beta-(\mathbb{A}+Q)_{-1})(\beta-
\mathbb{A}_{-1})^{-1}x+Q_\Lambda^\mathbb{A}
(\beta-\mathbb{A}_{-1})^{-1}x.
\end{eqnarray}

The property of
regularity implies
$D(L)=X_1+(\beta-\mathbb{A}_{-1})^{-1}\mathbb{B}\subset
D(Q^\mathbb{A}_\Lambda)$. Obviously, for any $x_0\in D(L)$,
$(\beta-\mathbb{A}_{-1})x_0\in
(\beta-A_{-1})D(Q^\mathbb{A}_\Lambda).$
Then (\ref{J}) implies that
\begin{eqnarray}\label{J1}
J^{\mathbb{A},\mathbb{A}+Q}(\beta-\mathbb{A}_{-1})x_0=(\beta-(\mathbb{A}
+Q)_{-1})x_0+Q_\Lambda^\mathbb{A}x_0.
\end{eqnarray}
On the other hand, there holds
\begin{align}\label{BG}
   \mathbb{B}Gx_0=Lx_0-\mathbb{A}_{-1}x_0, \forall x_0\in D(L).
\end{align}

Combine (\ref{J1}) and (\ref{BG}) to get
\begin{eqnarray*}
J^{\mathbb{A},\mathbb{A}+Q}\mathbb{B}Gx_0&=
&Lx_0-\beta x_0+J^{\mathbb{A},\mathbb{A}+Q}(\beta-\mathbb{A}_{-1})x_0\\
&=&Lx_0+Q_\Lambda^\mathbb{A}x_0-(\mathbb{A}+Q)_{-1}x_0\\
&=&Lx_0+Qx_0-(\mathbb{A}+Q)_{-1}x_0-\bar{Q}Gx_0\\
&=&\mathbb{B}_0Gx_0-\bar{Q}Gx_0, \forall x_0\in D(L).
\end{eqnarray*}
By definition, $G$ is surjection from $D(L)$ to $U$.
This indicates that
$$\mathbb{B}_0=J^{\mathbb{A},\mathbb{A}+Q}\mathbb{B}+\bar{Q}.$$
It follows from Theorem 3.9 of \cite{Mei2010} that
$(\mathbb{A}+Q,J^{\mathbb{A},\mathbb{A}+Q}\mathbb{B})$ generates an
abstract linear control system. The boundedness of $\bar{Q}$ implies
that $(\mathbb{A}+Q,\bar{Q})$ generates an abstract linear control
system with $\Phi_{\mathbb{A}+Q,\bar{Q}}(t)u=
\int_0^tT_{\mathbb{A}+Q}(t-s)\bar{Q}u(s)ds\in X,\ u\in L^p.$ So we
derive that
$(\mathbb{A}+Q,J^{\mathbb{A},\mathbb{A}+Q}\mathbb{B}+\bar{Q})$
generates an abstract linear control system with
\begin{align}\label{jia}
    \Phi_{\mathbb{A}+Q,J^{\mathbb{A},\mathbb{A}+Q}\mathbb{B}+\bar{Q}}
=\Phi_{\mathbb{A}+Q,J^{\mathbb{A},\mathbb{A}+Q}\mathbb{B}}
+\Phi_{\mathbb{A}+Q,\bar{Q}}.
\end{align}

Below we shall show that the mild expression (\ref{express}) of the
state $z(\cdot)$ holds. It follows from \cite{Hadd2006} that
$T_{\mathbb{A}+\mathbb{P}}(\cdot)x\in D^p(Q^\mathbb{A}_\Lambda).$
By the definition of state trajectory and (\ref{jia}), we have
\begin{align}\label{zdp}
   \nonumber z(\cdot)=&T_{\mathbb{A}+Q}(\cdot)x
+\Phi_{\mathbb{A}+Q,J^{\mathbb{A},\mathbb{A}+Q}\mathbb{B}
+\bar{Q}}(\cdot)u\\
=&T_{\mathbb{A}+Q}(\cdot)x+\Phi_{\mathbb{A}+Q,
J^{\mathbb{A},\mathbb{A}+Q}\mathbb{B}}(\cdot)u
+\Phi_{\mathbb{A}+Q,\bar{Q}}(\cdot)u\in D^p(Q^\mathbb{A}_\Lambda).
\end{align}

By \textbf{(S4)} and \cite[Theorem 3.1]{Hadd2006}, $(\mathbb{A}+Q,\bar{Q},-Q)$ generates a regular linear system with admissible feedback operator $I$. So
\begin{align}\label{phiapd}
    \Phi_{\mathbb{A}+Q,\mathbb{D}}-\Phi_{\mathbb{A},I}
F_{\mathbb{A}+Q,\bar{Q},Q}=\Phi_{\mathbb{A},\bar{Q}}.
\end{align}

Recall that we have proved in \cite[(4.4)]{Mei2010a} that
\begin{align}\label{kkk}
\Phi_{\mathbb{A}+Q,I}(t)F_{\mathbb{A},\mathbb{B},Q}(t)u
=\int_0^tT(t-s)Q_\Lambda^{\mathbb{A}}\Phi_{\mathbb{A}+Q,\mathbb{B}}(s)uds,\
u\in L^p(R^+,U).
\end{align}

Combine (\ref{jia}), (\ref{zdp}), (\ref{phiapd}) and (\ref{kkk}) to get
\begin{align*}
     z(t)=&T_{\mathbb{A}+Q}(t)x+\Phi_{\mathbb{A}+Q,
J^{\mathbb{A},\mathbb{A}+Q}\mathbb{B}+\bar{Q}}(t)u\\
=&T_{\mathbb{A}+Q}(t)x+\Phi_{\mathbb{A}+Q,J^{\mathbb{A},
\mathbb{A}+Q}\mathbb{B}}(t)u+\Phi_{\mathbb{A}+Q,\bar{Q}}(t)u\\
=&T_{\mathbb{A}+Q}(t)x+\Phi_{\mathbb{A}+Q,I}(t)
F_{\mathbb{A},\mathbb{B},Q}(t)u+\Phi_{\mathbb{A}+Q,\bar{Q}}(t)u\\
=&T_{\mathbb{A}}(t)x+\int_0^tT_{\mathbb{A}}(t-s)Q_\Lambda
T_{\mathbb{A}+Q}(s)xds
\\
&+\int_0^tT_{\mathbb{A}}(t-s)Q_\Lambda \Phi_{\mathbb{A}+Q,
J^{\mathbb{A},\mathbb{A}+Q}\mathbb{B}} ds+\Phi_{\mathbb{A}+Q,
\bar{Q}}(t)u+\Phi_{\mathbb{A},\mathbb{B}}(t)u\\
=&T_{\mathbb{A}}(t)x+\int_0^tT_{\mathbb{A}}(t-s)Q_\Lambda
T_{\mathbb{A}+Q}(s)xds
\\
&+\int_0^tT_{\mathbb{A}}(t-s)Q_\Lambda
[z(s)-T_{\mathbb{A}+Q}(\cdot)x-\Phi_{\mathbb{A}+Q,\bar{Q}}(\cdot)u]ds
+\Phi_{\mathbb{A}+Q,\bar{Q}}(t)u+\Phi_{\mathbb{A},\mathbb{B}}(t)u\\
=&T_{\mathbb{A}}(t)x+\int_0^tT_{\mathbb{A}}(t-s)Q_\Lambda
[z(s)-\Phi_{\mathbb{A}+\mathbb{P},\bar{Q}}(\cdot)u]ds
+\Phi_{\mathbb{A}+Q,\bar{Q}}(t)u+\Phi_{\mathbb{A},\mathbb{B}}(t)u\\
=&T_{\mathbb{A}}(t)x+\int_0^tT_{\mathbb{A}}(t-s)Q_\Lambda
z(s)ds+\Phi_{\mathbb{A}+Q,\bar{Q}}(t)u
-\Phi_{\mathbb{A},I}(t)F_{\mathbb{A}+Q,\bar{Q},Q}(t)u
+\Phi_{\mathbb{A},\mathbb{B}}(t)u\\
=&T_{\mathbb{A}}(t)x+\int_0^tT_{\mathbb{A}}(t-s)[Q_\Lambda
z(s)+\bar{Q}u(s)]ds +\Phi_{\mathbb{A},\mathbb{B}}(t)u.
\end{align*}
The proof is therefore completed.
\end{proof}

\begin{theorem}\label{exp}
Assume that the boundary systems $(\mathfrak{A}_m,\mathfrak{G},K)$,
$(\mathfrak{A}_m,\mathfrak{G},L)$ and $(A_m,P,M)$ are regular linear
systems. Then, system
\begin{eqnarray*}
  (BCS)\left\{

  \right.
\end{eqnarray*}
is a regular linear system with generator
$$\bigg(\left(
                            \begin{array}{cc}
                              \mathfrak{A} & 0 \\
                              L & A \\
                            \end{array}
                          \right),
J^{\left(
                            \begin{array}{cc}
                              \mathfrak{A} & 0 \\
                              0 & A \\
                            \end{array}
                          \right),
\left(
                            \begin{array}{cc}
                              \mathfrak{A} & 0 \\
                              L & A \\
                            \end{array}
                          \right)}\left(
                            \begin{array}{cc}
                              \mathfrak{B} & 0 \\
                              0 & B \\
                            \end{array}
                          \right)+\left(
           \begin{array}{cc}
             0 & 0 \\
             \bar{L} & 0 \\
           \end{array}
         \right),
                          \left(
                            \begin{array}{cc}
                              0 & I \\
                              K & M \\
                            \end{array}
                          \right),\left(
                \begin{array}{cc}
                  0 & 0 \\
                  \bar{K} & \bar{M} \\
                \end{array}
              \right)\bigg).$$
Moreover, the mild expression of the state is given by
\begin{align}\label{state}
    \left\{
      \begin{array}{ll}
        x(t)=T_{\mathfrak{A}}(t)x+\Phi_{\mathfrak{A},\mathfrak{B}}(t)u_1, & \hbox{} \\
        w(t)=T_A(t)w+\int_0^tT_A(t-s)[L_\Lambda x(s)+\bar{L}u_1(s)]ds
+\Phi_{A,B}(t)u_2, & \hbox{}
      \end{array}
    \right.
\end{align}
where $u_1$ and $u_2$ satisfy $u=\left(
                                   \begin{array}{c}
                                     u_1 \\
                                     u_2 \\
                                   \end{array}
                                 \right)
.$

\end{theorem}

\begin{proof}\ \
Denote by $\bar{K}$, $\bar{L}$ and $\bar{M}$ the feedthrough
operators of boundary systems $(\mathfrak{A}_m,\mathfrak{G},K)$,
$(\mathfrak{A}_m,\mathfrak{G},L)$ and $(A_m,P,M)$, respectively.
Rewrite system $(BCS)$ by
\begin{eqnarray*}
  \left\{
    %
  \right.
\end{eqnarray*}
Then the equation is of the form $(CS)$.

By assumption, it is obtained that $(\mathfrak{A},\mathfrak{B},K)$
generates a regular linear system and $(A,B)$ generates an abstract
linear control system. It follows from the proof of Theorem 3.4 in
\cite{Mei2015} that boundary control system $\bigg(\left(
                            \begin{array}{cc}
                              \mathfrak{A}_m & 0 \\
                              0 & A_m \\
                            \end{array}
                          \right),\left(
                            \begin{array}{cc}
                              \mathfrak{G} & 0 \\
                              0 & P \\
                            \end{array}
                          \right)\bigg)$ is
an abstract linear control system generated by
$\bigg(\left(
                            \begin{array}{cc}
                              \mathfrak{A} & 0 \\
                              0 & A \\
                            \end{array}
                          \right),
                          \left(
                            \begin{array}{cc}
                              \mathfrak{B} & 0 \\
                              0 & B \\
                            \end{array}
                          \right)\bigg)$ and
$$\Phi_{\left(
                            \begin{array}{cc}
                              \mathfrak{A} & 0 \\
                              0 & A \\
                            \end{array}
                          \right),\left(
                            \begin{array}{cc}
                              \mathfrak{B} & 0 \\
                              0 & B \\
                            \end{array}
                          \right)}=\left(
                                     \begin{array}{cc}
                                       \Phi_{\mathfrak{A},\mathfrak{B}} & 0 \\
                                       0 & \Phi_{A,B} \\
                                     \end{array}
                                   \right).
                          $$
Similarly, it is not hard to obtain that
$\bigg(\left(
                            \begin{array}{cc}
                              \mathfrak{A} & 0 \\
                              0 & A \\
                            \end{array}
                          \right),\left(
                            \begin{array}{cc}
                              0 & 0 \\
                              L & 0 \\
                            \end{array}
                          \right)\bigg)$
generates an abstract linear observation system with
\begin{align*}
    \Psi_{\left(
                            \begin{array}{cc}
                              \mathfrak{A} & 0 \\
                              0 & A \\
                            \end{array}
                          \right),\left(
                            \begin{array}{cc}
                              0 & 0 \\
                              L & 0 \\
                            \end{array}
                          \right)}=\left(
                                     \begin{array}{cc}
                                       0 & 0 \\
                                       \Psi_{\mathfrak{A},L} & 0 \\
                                     \end{array}
                                   \right).
\end{align*}
Denote $F_1=\left(
              \begin{array}{cc}
                0 & 0 \\
                F_{\mathfrak{A},\mathfrak{B},L} & 0 \\
              \end{array}
            \right).
$ We can test that $$\bigg(
\left(
  \begin{array}{cc}
    T_\mathfrak{A} & 0 \\
    0 & T_A \\
  \end{array}
\right),
\Phi_{\left(
                            \begin{array}{cc}
                              \mathfrak{A} & 0 \\
                              0 & A \\
                            \end{array}
                          \right),\left(
                            \begin{array}{cc}
                              \mathfrak{B} & 0 \\
                              0 & B \\
                            \end{array}
                          \right)},
\Psi_{\left(
                            \begin{array}{cc}
                              \mathfrak{A} & 0 \\
                              0 & A \\
                            \end{array}
                          \right),\left(
                            \begin{array}{cc}
                              0 & 0 \\
                              L & 0 \\
                            \end{array}
                          \right)},F_1\bigg)$$
is a regular linear system generated by
$\bigg(\left(
                            \begin{array}{cc}
                              \mathfrak{A} & 0 \\
                              0 & A \\
                            \end{array}
                          \right),\left(
                            \begin{array}{cc}
                              \mathfrak{B} & 0 \\
                              0 & B \\
                            \end{array}
                          \right),
                          \left(
                            \begin{array}{cc}
                              0 & 0 \\
                              L & 0 \\
                            \end{array}
                          \right)\bigg).$

We derive from Lemma \ref{output} that the boundary system
\begin{eqnarray}\label{bound}
  \left\{
    %
  \right.
\end{eqnarray}
is a regular linear system.
Since $\bar{L}$ is the feedthrough operator, $\lim_{\lambda\rightarrow
+\infty}L(\lambda-\mathfrak{A}_{-1})^{-1}\mathfrak{B}x=\bar{L}x$. We
compute the following limit
\begin{align*}
  &
\lim_{\lambda\rightarrow +\infty}\left(
                            \begin{array}{cc}
                              0 & 0 \\
                              L & 0 \\
                            \end{array}
                          \right)  \left(
                            \begin{array}{cc}
                              (\lambda-\mathfrak{A}_{-1})^{-1} & 0 \\
                              0 & (\lambda-A_{-1})^{-1} \\
                            \end{array}
                          \right)\left(
                            \begin{array}{cc}
                              \mathfrak{B} & 0 \\
                              0 & B \\
                            \end{array}
                          \right) \left(
                    \begin{array}{c}
                      x \\
                      w \\
                    \end{array}
                  \right)\\
=&\left(
                    \begin{array}{c}
                      0 \\
                      \lim_{\lambda\rightarrow
+\infty}L(\lambda-\mathfrak{A}_{-1})^{-1}\mathfrak{B}x \\
                    \end{array}
                  \right)\\
=&\left(
    \begin{array}{c}
      0 \\
      \bar{L}x \\
    \end{array}
  \right)
\end{align*}
to obtain that the feedthrough operator of system (\ref{bound}) is
 $\left(
           \begin{array}{cc}
             0 & 0 \\
             \bar{L} & 0 \\
           \end{array}
         \right)
$.
By Theorem \ref{boundarycontrol}, it follows that
$\bigg(\left(
                            \begin{array}{cc}
                              \mathfrak{A}_m & 0 \\
                              L & A_m \\
                            \end{array}
                          \right),\left(
                            \begin{array}{cc}
                              \mathfrak{G}_m & 0 \\
                              L & P \\
                            \end{array}
                          \right)\bigg)$ is an abstract linear control system generated by
$$\bigg(\left(
                            \begin{array}{cc}
                              \mathfrak{A} & 0 \\
                              L & A \\
                            \end{array}
                          \right),
J^{\left(
                            \begin{array}{cc}
                              \mathfrak{A} & 0 \\
                              0 & A \\
                            \end{array}
                          \right),
\left(
                            \begin{array}{cc}
                              \mathfrak{A} & 0 \\
                              L & A \\
                            \end{array}
                          \right)}\left(
                            \begin{array}{cc}
                              \mathfrak{B} & 0 \\
                              0 & B \\
                            \end{array}
                          \right)+\left(
           \begin{array}{cc}
             0 & 0 \\
             \bar{L} & 0 \\
           \end{array}
         \right)\bigg).$$

Denote $\Psi_1=\left(
                                     \begin{array}{cc}
                                       0 & \Psi_{A,I} \\
                                       \Psi_{\mathfrak{A},K} & \Psi_{A,M} \\
                                     \end{array}
                                   \right)
                         .$
It is not hard to derive that $\bigg(\left(
                                             \begin{array}{cc}
                                               T_\mathfrak{A} & 0 \\
                                               0 & T_A \\
                                             \end{array}
                                           \right)
,\Psi_1\bigg)$ is an abstract linear observation system generated by
$\bigg(\left(
                            \begin{array}{cc}
                              \mathfrak{A} & 0 \\
                              0 & A \\
                            \end{array}
                          \right),\left(
                            \begin{array}{cc}
                              0 & I \\
                              K & M \\
                            \end{array}
                          \right)\bigg).$
Moreover, we denote $F=\left(
               \begin{array}{cc}
                 0 & F_{A,B,I} \\
                 F_{\mathfrak{A},\mathfrak{B},K} & F_{A,B,M} \\
               \end{array}
             \right).
$ Then, by definition, it is not hard to test that
$$\bigg(
\left(
  \begin{array}{cc}
    T_\mathfrak{A} & 0 \\
    0 & T_A \\
  \end{array}
\right),
\Phi_{\left(
                            \begin{array}{cc}
                              \mathfrak{A} & 0 \\
                              0 & A \\
                            \end{array}
                          \right),\left(
                            \begin{array}{cc}
                              \mathfrak{B} & 0 \\
                              0 & B \\
                            \end{array}
                          \right)},
\Psi_{\left(
                            \begin{array}{cc}
                              \mathfrak{A} & 0 \\
                              0 & A \\
                            \end{array}
                          \right),\left(
                            \begin{array}{cc}
                              0 & I \\
                              K & M \\
                            \end{array}
                          \right)},F\bigg)$$
is a regular linear system
generated by $\bigg(\left(
                            \begin{array}{cc}
                              \mathfrak{A} & 0 \\
                              0 & A \\
                            \end{array}
                          \right),
                          \left(
                            \begin{array}{cc}
                              \mathfrak{B} & 0 \\
                              0 & B \\
                            \end{array}
                          \right),
                          \left(
                            \begin{array}{cc}
                              0 & I \\
                              K & M \\
                            \end{array}
                          \right)\bigg)$.

Combining Lemma \ref{observe} and Theorem \ref{boundarycontrol},
we obtain that
$$\bigg(\left(
                            \begin{array}{cc}
                              \mathfrak{A} & 0 \\
                              L & A \\
                            \end{array}
                          \right),
J^{\left(
                            \begin{array}{cc}
                              \mathfrak{A} & 0 \\
                              0 & A \\
                            \end{array}
                          \right),
\left(
                            \begin{array}{cc}
                              \mathfrak{A} & 0 \\
                              L & A \\
                            \end{array}
                          \right)}\left(
                            \begin{array}{cc}
                              \mathfrak{B} & 0 \\
                              0 & B \\
                            \end{array}
                          \right),
                          \left(
                            \begin{array}{cc}
                              0 & I \\
                              K & M \\
                            \end{array}
                          \right)\bigg)$$
is a regular linear system.
By \textbf{(S4)}, the boundedness of the operator $\bar{L}$ implies that
$$\bigg(\left(
                            \begin{array}{cc}
                              \mathfrak{A} & 0 \\
                              L & A \\
                            \end{array}
                          \right),
J^{\left(
                            \begin{array}{cc}
                              \mathfrak{A} & 0 \\
                              0 & A \\
                            \end{array}
                          \right),
\left(
                            \begin{array}{cc}
                              \mathfrak{A} & 0 \\
                              L & A \\
                            \end{array}
                          \right)}\left(
                            \begin{array}{cc}
                              \mathfrak{B} & 0 \\
                              0 & B \\
                            \end{array}
                          \right)+\left(
           \begin{array}{cc}
             0 & 0 \\
             \bar{L} & 0 \\
           \end{array}
         \right),
                          \left(
                            \begin{array}{cc}
                              0 & I \\
                              K & M \\
                            \end{array}
                          \right)\bigg)$$
generates a regular linear system.
It follows from Lemma \ref{output} that
$(BCS)$ is a regular linear system and the feedthrough operator $\mathbb{D}$ satisfies
\begin{align*}
    \mathbb{D}\left(
                \begin{array}{cc}
                  \mathfrak{G} & 0 \\
                  0 & P \\
                \end{array}
              \right)
\left(
                \begin{array}{c}
                  x \\
                  w \\
                \end{array}
              \right)=\left(
                            \begin{array}{cc}
                              0 & I \\
                              K & M \\
                            \end{array}
                          \right)\left(
                \begin{array}{c}
                  x \\
                  w \\
                \end{array}
              \right)-\left(
                            \begin{array}{cc}
                              0 & I \\
                              K & M \\
                            \end{array}
                          \right)^{\left(
                            \begin{array}{cc}
                              \mathfrak{A} & 0 \\
                              L & A \\
                            \end{array}
                          \right)}_\Lambda\left(
                \begin{array}{c}
                  x \\
                  w \\
                \end{array}
              \right),
\end{align*}
$\left(
                \begin{array}{c}
                  x \\
                  w \\
                \end{array}
              \right)\in D(\mathfrak{A}_m)\times D(A_m)$.
We compute
\begin{align*}
 &\left(
                            \begin{array}{cc}
                              0 & I \\
                              K & M \\
                            \end{array}
                          \right)^{\left(
                            \begin{array}{cc}
                              \mathfrak{A} & 0 \\
                              L & A \\
                            \end{array}
                          \right)}_\Lambda\left(
                \begin{array}{c}
                  x \\
                  w \\
                \end{array}
              \right)\\
=&
\left(
                            \begin{array}{cc}
                              0 & I \\
                              K & M \\
                            \end{array}
                          \right)^{\left(
                            \begin{array}{cc}
                              \mathfrak{A} & 0 \\
                              0 & A \\
                            \end{array}
                          \right)}_\Lambda\left(
                \begin{array}{c}
                  x \\
                  w \\
                \end{array}
              \right)\\
=&\lim_{\lambda\rightarrow +\infty}
\lambda\left(
                            \begin{array}{cc}
                              0 & I \\
                              K & M \\
                            \end{array}
                          \right)
\left(
                            \begin{array}{cc}
                              (\lambda-\mathfrak{A})^{-1} & 0 \\
                              0 & (\lambda-A)^{-1} \\
                            \end{array}
                          \right)\left(
                \begin{array}{c}
                  x \\
                  w \\
                \end{array}
              \right)\\
=&\left(
                \begin{array}{c}
                  w \\
                 K_\Lambda x+ M_\Lambda w \\
                \end{array}
              \right).
\end{align*}
By Lemma \ref{output}, $Kx-K_\Lambda x=\bar{K}Gx$ and
$Mw-M_\Lambda w=\bar{M}Pw$. We derive
\begin{align*}
    \mathbb{D}\left(
                \begin{array}{cc}
                  G & 0 \\
                  0 & P \\
                \end{array}
              \right)
\left(
                \begin{array}{c}
                  x \\
                  w \\
                \end{array}
              \right)=&\left(
                \begin{array}{c}
                  0 \\
                  Kx-K_\Lambda x+Mw-M_\Lambda w\\
                \end{array}
              \right)\\
=&
\left(
                \begin{array}{cc}
                  0 & 0 \\
                  \bar{K} & \bar{M} \\
                \end{array}
              \right)\left(
                \begin{array}{cc}
                  \mathfrak{G} & 0 \\
                  0 & P \\
                \end{array}
              \right)
\left(
                \begin{array}{c}
                  x \\
                  w \\
                \end{array}
              \right).
\end{align*}
Since $G$ and $P$ are surjective, the feedthrough
operator $\mathbb{D}=\left(
                \begin{array}{cc}
                  0 & 0 \\
                  \bar{K} & \bar{M} \\
                \end{array}
              \right)$.

It is not had to test that $\left(
  \begin{array}{cc}
    0 & 0 \\
    L & 0 \\
  \end{array}
\right)^{\left(
            \begin{array}{cc}
              \mathfrak{A} & 0 \\
              0 & A \\
            \end{array}
          \right)}_\Lambda=\left(
  \begin{array}{cc}
    0 & 0 \\
    L_\Lambda^\mathfrak{A} & 0 \\
  \end{array}
\right)$. By Theorem \ref{boundarycontrol}, the mild expression of the state
 $\left(
    \begin{array}{c}
      x(\cdot) \\
      w(\cdot) \\
    \end{array}
  \right)
$ is as follows
\begin{align*}
 &\left(
    \begin{array}{c}
      x(t) \\
      w(t) \\
    \end{array}
  \right)\\
=&\left(
            \begin{array}{cc}
              T_\mathfrak{A}(t) & 0 \\
              0 & T_{A}(t) \\
            \end{array}
          \right)\left(
    \begin{array}{c}
      x \\
      w \\
    \end{array}
  \right)
+\int_0^t\left(
            \begin{array}{cc}
              T_\mathfrak{A}(t-s) & 0 \\
              0 & T_{A}(t-s) \\
            \end{array}
          \right)\bigg[
\left(
  \begin{array}{cc}
    0 & 0 \\
    L_\Lambda^\mathfrak{A} & 0 \\
  \end{array}
\right)
\left(
  \begin{array}{c}
    x(s) \\
    w(s) \\
  \end{array}
\right)\\
&+\left(
  \begin{array}{cc}
    0 & 0 \\
    \bar{L} & 0 \\
  \end{array}
\right)\left(
         \begin{array}{c}
           u_1(s) \\
           u_2(s) \\
         \end{array}
       \right)
\bigg]
ds+\left(
            \begin{array}{cc}
              \Phi_{\mathfrak{A},\mathfrak{B}}(t) & 0 \\
              0 & \Phi_{A,B}(t) \\
            \end{array}
          \right)\left(
         \begin{array}{c}
           u_1 \\
           u_2 \\
         \end{array}
       \right),
\end{align*}
which indicates that (\ref{state}) holds.
The proof is completed.
\end{proof}

In order to prove the well-posedness of system $(ABS)$,
the following Lemma should be studied.

\begin{lemma}\label{pertabation}
Assume that the boundary system $(L,G,Q)$ is a regular linear system
generated by $(\mathbb{A},\mathbb{B},\mathbb{C},\mathbb{D})$ on
$(X,U,U)$ with admissible feedback operator $I$. Then system
\begin{align*}
  \left\{
   \begin{array}{ll}
   \dot{z}(t)&=Lz(t)\\
     Gz(t)&=Qz(t)
   \end{array}
 \right.
\end{align*}
is well-posed, which is equivalent to
that $\mathbb{A}^I$ generates a $C_0$-semigroup.
\end{lemma}
\begin{proof}\ \
Denote $\mathbb{A}^Q:=L$ with domain $D(\mathbb{A}^Q)=\{x\in Z:Gx=Qx\}.$
Since $(\mathbb{A},\mathbb{B},\mathbb{C},\mathbb{D})$ are regular linear system with admissible feedback operator $I$, by Theorem \ref{relation} it follows that
the system operator $\mathbb{A}^I$ of the closed loop system generates a $C_0$-semigroup and $\mathbb{A}^I=\mathbb{A}_{-1}+
\mathbb{B}(I-\mathbb{D})^{-1}_{Left}\mathbb{C}_\Lambda^\mathbb{A},$
with domain $D(\mathbb{A}^I)=\{x\in D(\mathbb{C}^\mathbb{A}_\Lambda):\mathbb{A}_{-1}x+
\mathbb{B}(I-\mathbb{D})^{-1}_{Left}\mathbb{C}_\Lambda^\mathbb{A}x\in X\}\subset D(\mathbb{C}_\Lambda^\mathbb{A}).$

Our aim is to show that $\mathbb{A}^I=\mathbb{A}^Q.$ For any $x\in D(\mathbb{A}^I)$, there holds
\begin{align*}
    \mathbb{A}^Ix=&\mathbb{A}_{-1}x+\mathbb{B}(I-\mathbb{D})^{-1}_{Left}\mathbb{C}_\Lambda^\mathbb{A}x\\
=&\mathbb{A}_{-1}\big(x-R(\lambda,\mathbb{A}_{-1})\mathbb{B}(I-\mathbb{D})^{-1}_{Left}
\mathbb{C}_\Lambda^\mathbb{A}x\big)+\lambda R(\lambda,\mathbb{A}_{-1})\mathbb{B}
(I-\mathbb{D})^{-1}_{Left}\mathbb{C}_\Lambda^\mathbb{A}x\in X.
\end{align*}
This implies that $x-R(\lambda,\mathbb{A}_{-1})\mathbb{B}(I-\mathbb{D})^{-1}_{Left}
\mathbb{C}_\Lambda^\mathbb{A}x\in D(\mathbb{A}).$
By (\ref{GB}), $Gx=GR(\lambda,\mathbb{A}_{-1})\mathbb{B}(I-\mathbb{D})^{-1}_{Left}
\mathbb{C}_\Lambda^\mathbb{A}x=(I-\mathbb{D})^{-1}_{Left}
\mathbb{C}_\Lambda^\mathbb{A}x.$
The combining of this and Lemma \ref{output} indicates that $$Qx=\mathbb{C}_\Lambda^\mathbb{A}x+\mathbb{D}Gx
=\mathbb{C}_\Lambda^\mathbb{A}x+\mathbb{D}(I-\mathbb{D})^{-1}_{Left}
\mathbb{C}_\Lambda^\mathbb{A}x=(I-\mathbb{D})^{-1}_{Left}
\mathbb{C}_\Lambda^\mathbb{A}x=Gx.$$
Then $\mathbb{A}^Ix=\mathbb{A}_{-1}x+\mathbb{B}(I-\mathbb{D})^{-1}_{Left}\mathbb{C}_\Lambda^\mathbb{A}x=Lx.$
Hence $\mathbb{A}^I\subset \mathbb{A}^Q.$

For $x\in D(\mathbb{A}^Q)$, $\mathbb{A}^Qx=Lx=A_{-1}x+\mathbb{B}Gx
$ and $Gx=\mathbb{C}_\Lambda^\mathbb{A}x
+\mathbb{D}Gx.$ That $I$ is an admissible feedback operator for regular linear system $(\mathbb{A},\mathbb{B},\mathbb{C},\mathbb{D})$
implies that $(I-\mathbb{D})^{-1}_{Left}$ is invert.
So $Gx=(I-\mathbb{D})^{-1}_{Left}\mathbb{C}_\Lambda^\mathbb{A}x$,
thereby $\mathbb{A}^Qx=A_{-1}x+\mathbb{B}(I-\mathbb{D})^{-1}_{Left}\mathbb{C}_\Lambda^\mathbb{A}x.$
This implies that $\mathbb{A}^Q\subset \mathbb{A}^I.$
The proof is therefore completed.
\end{proof}
\begin{remark}
In the special case that the feedthough operator is zero, the Lemma has been proved by Hadd \cite{Hadd2008}.
Our Lemma is the more generalized case and our proof is stimulated by \cite{Hadd2008}.
\end{remark}

\begin{lemma}\label{invert}
Let $E$ and $F$ be Banach spaces; $A\in L(E,E)$, $B\in L(F,E)$, $C\in L(E,F)$, and $D\in L(F,F)$. Assume that $A$ and $D$ are invertible, and $A-BD^{-1}C$ is also invertible. Then the matrix operator $\left(
   \begin{array}{cc}
     A & B \\
     C & D \\
   \end{array}
 \right)
$ is invertible.
\end{lemma}
\begin{proof}\ \
We compute
\begin{align*}
   &(D-CA^{-1}B)[D^{-1}+D^{-1}C(A-BD^{-1}C)^{-1}BD^{-1}]\\
   =& I+C(A-BD^{-1}C)^{-1}BD^{-1}-CA^{-1}BD^{-1}\\
   & -CA^{-1}[(BD^{-1}C-A)+A](A-BD^{-1}C)^{-1}BD^{-1}\\
   =& I
\end{align*}
and
\begin{align*}
    &[D^{-1}+D^{-1}C(A-BD^{-1}C)^{-1}BD^{-1}](D-CA^{-1}B)\\
    =& I-D^{-1}CA^{-1}B+D^{-1}C(A-BD^{-1}C)^{-1}B \\
     &-D^{-1}C(A-BD^{-1}C)^{-1}[(BD^{-1}C-A)+A]A^{-1}B\\
     =& I.
\end{align*}
Then $D-CA^{-1}B$ is invertible and $(D-CA^{-1}B)^{-1}=D^{-1}+D^{-1}C(A-BD^{-1}C)^{-1}BD^{-1}.$
Through simple calculation, we can see that
$\left(
   \begin{array}{cc}
     A & B \\
     C & D \\
   \end{array}
 \right)
$ is invertible and
$\left(
   \begin{array}{cc}
     A & B \\
     C & D \\
   \end{array}
 \right)^{-1}=
 \left(
   \begin{array}{cc}
     (A-BD^{-1}C)^{-1} & -A^{-1}B(D-CA^{-1}B)^{-1} \\
     -D^{-1}C(A-BD^{-1}C)^{-1} & (D-CA^{-1}B)^{-1} \\
   \end{array}
 \right).
$
The proof is completed.
\end{proof}

\begin{theorem}\label{feedback}
Assume that $(\mathfrak{A}_m,\mathfrak{G},K)$ and
$(\mathfrak{A}_m,\mathfrak{G},L)$ generate regular linear systems.
Suppose $(A_m,P,M)$ to generate a regular linear system with
admissible feedback operator $I$.
 Then system $(ABS)$ is well-posed.
\end{theorem}
\begin{proof}\ \
Denote $\mathbb{A}=\left(
                            \begin{array}{cc}
                              \mathfrak{A} & 0 \\
                              0 & A \\
                            \end{array}
                          \right)$,
$\mathbb{B}=\left(
                            \begin{array}{cc}
                              \mathfrak{B} & 0 \\
                              0 & B \\
                            \end{array}
                          \right)$,
$\mathbb{C}=\left(
                            \begin{array}{cc}
                              0 & I \\
                              K & M \\
                            \end{array}
                          \right)$,
$\mathbb{P}=\left(
                            \begin{array}{cc}
                              0 & 0 \\
                              L & 0 \\
                            \end{array}
                          \right).$
Then $\mathbb{A}$, $\mathbb{B}$, $\mathbb{C}$ and $\mathbb{P}$
satisfies the assumption of Theorem \ref{observe} and
$\bar{\mathbb{P}}=\left(
                            \begin{array}{cc}
                              0 & 0 \\
                              \bar{L} & 0 \\
                            \end{array}
                          \right)$.
It is not hard to show that
\begin{align}\label{AIC}
    F_{\mathbb{A},I,\mathbb{C}}
  =\left(
     \begin{array}{cc}
       0 & F_{A,I,I} \\
       F_{\mathfrak{A},I,K} & F_{A,I,M} \\
     \end{array}
   \right),
\end{align}
and
\begin{align}\label{AIP}
    F_{\mathbb{A},I,\mathbb{P}}=\left(
     \begin{array}{cc}
       0 & 0 \\
       F_{\mathfrak{A},I,L} & 0 \\
     \end{array}
   \right).
\end{align}
By the proof of Theorem \ref{boundarycontrol}, it follows that
\begin{align}\label{ABP}
  F_{\mathbb{A},\mathbb{B},\mathbb{P}}=\left(
              \begin{array}{cc}
                0 & 0 \\
                F_{\mathfrak{A},\mathfrak{B},L} & 0 \\
              \end{array}
            \right),
\end{align}
\begin{align}\label{ABP}
  F_{\mathbb{A},\bar{\mathbb{P}},\mathbb{P}}=0,
\end{align}
\begin{align}\label{ABC}
  F_{\mathbb{A},\mathbb{B},\mathbb{C}}=\left(
               \begin{array}{cc}
                 0 & F_{A,B,I} \\
                 F_{\mathfrak{A},\mathfrak{A},K} & F_{A,B,M} \\
               \end{array}
             \right),
\end{align}
and
\begin{align}\label{ABC}
  F_{\mathbb{A},\bar{\mathbb{P}},\mathbb{C}}=\left(
               \begin{array}{cc}
                 F_{A,\bar{L},I} & 0 \\
                 F_{A,\bar{L},M} & 0 \\
               \end{array}
             \right).
\end{align}

Substitute (\ref{AIC}), (\ref{AIP}), (\ref{ABP}) and (\ref{ABC}) into (\ref{main}) to get
\begin{align*}
&F_{\mathbb{A}+\mathbb{P},J^{\mathbb{A},\mathbb{A}
+\mathbb{P}}\mathbb{B},\mathbb{C}}\\
=&\left(
     \begin{array}{cc}
       0 & F_{A,I,I} \\
       F_{\mathfrak{A},I,K} & F_{A,I,M} \\
     \end{array}
   \right)
   \bigg(I-\left(
     \begin{array}{cc}
       0 & 0 \\
       F_{\mathfrak{A},I,L} & 0 \\
     \end{array}
   \right)\bigg)^{-1}
\left(
              \begin{array}{cc}
                0 & 0 \\
                F_{\mathfrak{A},\mathfrak{B},L} & 0 \\
              \end{array}
            \right)
+\left(
               \begin{array}{cc}
                 0 & F_{A,B,I} \\
                 F_{\mathfrak{A},\mathfrak{B},K} &
F_{A,B,M} \\
               \end{array}
             \right)\\
=& \left(
               \begin{array}{cc}
                 F_{A,I,I}F_{\mathfrak{A},\mathfrak{B},L} & F_{A,B,I} \\
                 F_{A,I,M}F_{\mathfrak{A},\mathfrak{B},L}+F_{\mathfrak{A},\mathfrak{B},K}&
F_{A,B,M} \\
               \end{array}
             \right).
\end{align*}
The following holds
\begin{align*}
   F_{\mathbb{A}+\mathbb{P},\bar{\mathbb{P}},\mathbb{C}}
=F_{\mathbb{A},I,\mathbb{C}}(I-F_{\mathbb{A},I,\mathbb{P}})^{-1}
F_{\mathbb{A},\bar{\mathbb{P}},\mathbb{P}}+F_{\mathbb{A},\bar{\mathbb{P}},\mathbb{C}}
=\left(
               \begin{array}{cc}
                 F_{A,\bar{L},I} & 0 \\
                 F_{A,\bar{L},M} & 0 \\
               \end{array}
             \right).
\end{align*}

Denote $\mathbb{M}=\left(
                       \begin{array}{cc}
                         0 & 0 \\
                         \bar{K} & \bar{M} \\
                       \end{array}
                     \right)$. By \textbf{(S1)}, the transform function of system ($BCS$) satisfies
\begin{align*}
    &(F_{\mathbb{A}+\mathbb{P},J^{\mathbb{A},\mathbb{A}
+\mathbb{P}}\mathbb{B}+\bar{\mathbb{P}},\mathbb{C},
\mathbb{M}}(s)u)(t)\\
=&(F_{\mathbb{A}+\mathbb{P},J^{\mathbb{A},\mathbb{A}
+\mathbb{P}}\mathbb{B}+\bar{\mathbb{P}},\mathbb{C}
}(s)u)(t)+\mathbb{M}u(t)\\
=&
(F_{\mathbb{A}+\mathbb{P},J^{\mathbb{A},\mathbb{A}
+\mathbb{P}}\mathbb{B},\mathbb{C}}(s)u)(t)+
(F_{\mathbb{A}+\mathbb{P},\bar{\mathbb{P}},\mathbb{C}}(s)u)(t)
+\mathbb{M}u(t)\\
=&\bigg(\left(
               \begin{array}{cc}
                 F_{A,I,I}F_{A,B,L}+F_{A,\bar{L},I} &
F_{A,B,I} \\
                 F_{A,I,M}F_{\mathfrak{A},\mathfrak{B},L}+
F_{\mathfrak{A},\mathfrak{B},K}
+F_{A,\bar{L},M}& F_{A,B,M} \\
               \end{array}
             \right)(s)u\bigg)(t)+\left(
                       \begin{array}{cc}
                         0 & 0 \\
                         \bar{K} & \bar{M} \\
                       \end{array}
                     \right)u(t)\\
=&\bigg(\left(
               \begin{array}{cc}
                 F_{A,I,I}F_{\mathfrak{A},\mathfrak{B},L}+F_{A,\bar{L},I} &
F_{A,B,I} \\
                 F_{A,I,M}F_{\mathfrak{A},\mathfrak{B},L}+
F_{\mathfrak{A},\mathfrak{B},K,\bar{K}}
+F_{A,\bar{L},M}& F_{A,B,M,\bar{M}} \\
               \end{array}
             \right)(s)u\bigg)(t).
\end{align*}
Hence $$F_{\mathbb{A}+\mathbb{P},J^{\mathbb{A},\mathbb{A}
+\mathbb{P}}\mathbb{B}+\bar{\mathbb{P}},\mathbb{C},
\mathbb{M}}=\left(
               \begin{array}{cc}
                 F_{A,I,I}F_{\mathfrak{A},\mathfrak{B},L}+F_{A,\bar{L},I} &
F_{A,B,I} \\
                 F_{A,I,M}F_{\mathfrak{A},\mathfrak{B},L}+
F_{\mathfrak{A},\mathfrak{B},K,\bar{K}}
+F_{A,\bar{L},M}& F_{A,B,M,\bar{M}} \\
               \end{array}
             \right).$$

This implies that
\begin{align*}
 I-F_{\mathbb{A}+\mathbb{P},J^{\mathbb{A},\mathbb{A}
+\mathbb{P}}\mathbb{B}+\bar{\mathbb{P}},\mathbb{C},\mathbb{M}}=\left(
               \begin{array}{cc}
                 I-(F_{A,I,I}F_{\mathfrak{A},\mathfrak{B},L}+F_{A,\bar{L},I}) &
-F_{A,B,I} \\
                 -F_{A,I,M}F_{\mathfrak{A},\mathfrak{B},L}-
F_{\mathfrak{A},\mathfrak{B},K,\bar{K}}-F_{A,\bar{L},M} &
I-F_{A,B,M,\bar{M}} \\
               \end{array}
             \right).
\end{align*}
By \textbf{(S3)}, it follows that
\begin{align}\label{ri0}
   \|F_{A,B,I}(t)\|\rightarrow 0, \ \|F_{A,I,I}(t)\|\rightarrow 0, \ \|F_{A,\bar{L},I}(t)\|\rightarrow 0,
\end{align}
as $t\rightarrow 0$.
It follows that $$\|F_{A,I,I}(t)F_{\mathfrak{A},\mathfrak{B},L}(t)+F_{A,\bar{L},I}(t)\|\leq \|F_{A,I,I}(t)\|\|F_{\mathfrak{A},\mathfrak{B},L}(t)\|+\|F_{A,\bar{L},I}(t)\|\rightarrow 0$$ as $t\rightarrow 0,$
which implies that $I-(F_{A,I,I}F_{\mathfrak{A},\mathfrak{B},L}+F_{A,\bar{L},I})$ is invertible as $t\rightarrow 0.$
Since $I$ is admissible feedback operator for regular linear system $(A,B,M,\bar{M})$, $I-F_{A,B,M,\bar{M}}(t)$ is invertible as for any $t\geq 0$ and the transfer function $$F_{A,B,M,\bar{M}}(\cdot)[I-F_{A,B,M,\bar{M}}(\cdot)]^{-1}
=[I-F_{A,B,M,\bar{M}}(\cdot)]^{-1}-I$$ is bounded on any bounded interval.
This implies that $I-F_{A,B,M,\bar{M}}(\cdot)$ is bounded on any bounded interval. Combine this with (\ref{ri0}) to get that
\begin{align*}
  I-(F_{A,I,I}F_{\mathfrak{A},\mathfrak{B},L}+F_{A,\bar{L},I})
  -F_{A,B,I}(I-F_{A,B,M,\bar{M}})^{-1}[F_{A,I,M}F_{\mathfrak{A},\mathfrak{B},L}
+F_{\mathfrak{A},\mathfrak{B},K,\bar{K}}+F_{A,\bar{L},M}]
\end{align*}
is invertible.
By Lemma \ref{invert}, $I-F_{\mathbb{A}+\mathbb{P},J^{\mathbb{A},\mathbb{A}
+\mathbb{P}}\mathbb{B}+\bar{\mathbb{P}},\mathbb{C},\mathbb{M}}(\cdot)$
is invertible for enough small $t$. By Lemma \ref{pertabation}, system $(ABS)$ is well-posed.
\end{proof}

\begin{theorem}\label{8}
Assume that $(\mathfrak{A}_m,\mathfrak{G},K)$ and
$(\mathfrak{A}_m,\mathfrak{G},L)$ generate regular linear systems.
Suppose $(A_m,P,M)$ to generate a regular linear system with
admissible feedback operator $I$. Suppose that
$\lambda\in\rho(A)\cap\rho(A+L e_\lambda)$ and $1\in
\rho(MR(\lambda,A_{-1})B)$. Then
\begin{eqnarray*}
&&\lambda\in \sigma_P(\mathcal {A}_{L,K,M})\\
&\Longleftrightarrow& 1\in
\sigma_P\big(R(\lambda,A)Le_\lambda+R(\lambda,A_{-1})B
(I-MR(\lambda,A_{-1})B)^{-1}[MR(\lambda,A)Le_\lambda
+Ke_\lambda]\big)\\
&\Longleftrightarrow& 1\in
\sigma_P\big(MR(\lambda,A_{-1})B+(MR(\lambda,A)Le_\lambda
+Ke_\lambda)(I-R(\lambda,A)Le_\lambda)^{-1}R(\lambda,A_{-1})B\big).
\end{eqnarray*}
\end{theorem}

\begin{proof}\ \
By \cite[Proposition 1]{Hadd2008}, it follow that for $\lambda \in
\rho\bigg(\left(
         \begin{array}{cc}
           \mathfrak{A} & 0 \\
           L &  A\\
         \end{array}
       \right)
\bigg)$, $\lambda\in \sigma_P(\mathcal {A}_{L,K,M})$ if and only if
$$1\in \sigma_P(G_{\mathbb{A}+\mathbb{P},J^{\mathbb{A},\mathbb{A}
+\mathbb{P}}\mathbb{B}+\bar{\mathbb{P}},\mathbb{C},\mathbb{M}}(\lambda)
),$$ where
\begin{align*}
   G_{\mathbb{A}+\mathbb{P},J^{\mathbb{A},\mathbb{A}
+\mathbb{P}}\mathbb{B}+\bar{\mathbb{P}},\mathbb{C},\mathbb{M}}(\lambda)
=&\left(
               \begin{array}{cc}
                 G_{A,I,I}(\lambda)G_{\mathfrak{A},\mathfrak{B},L}(\lambda)
                 +G_{A,\bar{L},I}(\lambda) &
G_{A,B,I}(\lambda) \\
                G_{A,I,M}(\lambda)G_{\mathfrak{A},\mathfrak{B},L}(\lambda)+
G_{\mathfrak{A},\mathfrak{B},K,\bar{K}}(\lambda)
+G_{A,\bar{L},M}(\lambda)& G_{A,B,M,\bar{M}}(\lambda) \\
               \end{array}
             \right)\\
=&\left(
    \begin{array}{cc}
      R(\lambda,A)L e_\lambda & R(\lambda,A_{-1})B \\
      MR(\lambda,A)L e_\lambda+Ke_\lambda & MR(\lambda,A_{-1})B\\
    \end{array}
  \right)
.
\end{align*}

Obviously, $\lambda\in\rho(\mathfrak{A})$, thereby, $\lambda \in
\rho\bigg(\left(
         \begin{array}{cc}
           \mathfrak{A} & 0 \\
           L & A \\
         \end{array}
       \right)
\bigg)$ if and only if $\lambda\in\rho(A)$,
and in this case
 $$R\bigg(\lambda,\left(
         \begin{array}{cc}
           \mathfrak{A} & 0 \\
           L &  A\\
         \end{array}
       \right)
\bigg)=\left(
         \begin{array}{cc}
           R(\lambda,\mathfrak{A}) & 0 \\
           R(\lambda,A)LR(\lambda,\mathfrak{A}) & R(\lambda,A) \\
         \end{array}
       \right).$$
It is not hard to prove that $1\in \rho(R(\lambda,A)L e_\lambda)\Leftrightarrow \lambda\in\rho(A+L e_\lambda)$.

By definition, $1\in \sigma_P(G_{\mathbb{A}+\mathbb{P},J^{\mathbb{A},\mathbb{A}
+\mathbb{P}}\mathbb{B}+\bar{\mathbb{P}},\mathbb{C},\mathbb{M}}(\lambda)
)$ if and only if the following equation has nonzero solution
\begin{align}\label{func}
    \bigg(I-\left(
    \begin{array}{cc}
      R(\lambda,A)L e_\lambda & R(\lambda,A_{-1})B \\
      MR(\lambda,A)L e_\lambda+Ke_\lambda & MR(\lambda,A_{-1})B\\
    \end{array}
  \right)\bigg) \left(
                 \begin{array}{c}
                   x \\
                   f \\
                 \end{array}
               \right)=0,
\end{align}
which is equivalent to that
\begin{align*}
    \left\{
      \begin{array}{ll}
       (I-R(\lambda,A)Le_\lambda)x=R(\lambda,A_{-1})Bf, & \hbox{} \\
        (MR(\lambda,A)L e_\lambda+Ke_\lambda)x=(I-MR(\lambda,A_{-1})B)f. & \hbox{}
      \end{array}
    \right.
\end{align*}
has nonzero solution.
The equivalence relations of this theorem are obtained because
both $(I-R(\lambda,A)Le_\lambda)$ and $(I-MR(\lambda,A_{-1})B)$
are invertible.
\end{proof}

\begin{theorem}\label{yujie}
Assume that $(\mathfrak{A}_m,\mathfrak{G},K)$ and
$(\mathfrak{A}_m,\mathfrak{G},L)$ generate regular linear systems.
Suppose $(A_m,P,M)$ to generate a regular linear system with
admissible feedback operator $I$. Suppose that
$\lambda\in\rho(A)\cap\rho(A+L e_\lambda)$ and $1\in
\rho(MR(\lambda,A_{-1})B)$. Then
\begin{eqnarray*}
&&\lambda\in \rho(\mathcal {A}_{L,K,M})\\
&\Longleftarrow& 1\in
\rho\big(R(\lambda,A)Le_\lambda+R(\lambda,A_{-1})B
(I-MR(\lambda,A_{-1})B)^{-1}[MR(\lambda,A)Le_\lambda
+Ke_\lambda]\big)\\
&\Longleftrightarrow& 1\in
\rho\big(MR(\lambda,A_{-1})B+(MR(\lambda,A)Le_\lambda
+Ke_\lambda)(I-R(\lambda,A)Le_\lambda)^{-1}R(\lambda,A_{-1})B\big).
\end{eqnarray*}
Moreover, in this case,
\begin{align*}
    &R(\lambda,\mathcal {A}_{L,K,M})\\
=&
\left(
  \begin{array}{cc}
    R(\lambda,\mathfrak{A})+e_\lambda W_1 & e_\lambda W_2 \\
     W_5& W_6 \\
  \end{array}
\right),
\end{align*}
where, $$W_5=R(\lambda,A)LR(\lambda,\mathfrak{A})+R(\lambda,A)Le_\lambda W_1+R(\lambda,A_{-1}B)W_3,$$
$$W_6=R(\lambda,A)+R(\lambda,A)Le_\lambda W_2+R(\lambda,A_{-1}B)W_4,$$
$$W_1=N_1R(\lambda,A)LR(\lambda,\mathfrak{A})
+(I-R(\lambda,A)Le_\lambda)^{-1}R(\lambda,A_{-1})BN_2
[KR(\lambda,\mathfrak{A})+MR(\lambda,A)LR(\lambda,\mathfrak{A})],$$
\begin{align*}
    W_2=N_1R(\lambda,A)
+(I-R(\lambda,A)Le_\lambda)^{-1}R(\lambda,A_{-1})BN_2
MR(\lambda,A),
\end{align*}
\begin{align*}
    W_3=&(I-MR(\lambda,A_{-1})B)^{-1}(MR(\lambda,A)Le_\lambda+Ke_\lambda)
N_1R(\lambda,A)LR(\lambda,\mathfrak{A})\\
&+N_2(MR(\lambda,A)LR(\lambda,\mathfrak{A})+KR(\lambda,\mathfrak{A})),
\end{align*}
\begin{align*}
    W_4=(I-MR(\lambda,A_{-1})B)^{-1}(MR(\lambda,A)Le_\lambda+Ke_\lambda)
N_1R(\lambda,A)+N_2MR(\lambda,A),
\end{align*}
\begin{align*}
    N_1=\{I-R(\lambda,A)Le_\lambda-R(\lambda,A_{-1})B(I-MR(\lambda,A_{-1})B)^{-1}[MR(\lambda,A)Le_\lambda+Ke_\lambda]\}^{-1},
\end{align*}
\begin{align*}
    N_2=\{I-MR(\lambda,A_{-1})B-[MR(\lambda,A)Le_\lambda+Ke_\lambda](I-R(\lambda,A)Le_\lambda)^{-1}R(\lambda,A_{-1})B\}^{-1}.
\end{align*}

\end{theorem}

\begin{proof}\ \
By \cite{Weiss1994b}, it follows that
\begin{align*}
    R(\lambda,\mathcal {A}_{L,K,M})
    =&R(\lambda,\mathbb{A}+\mathbb{P})
+
    R(\lambda,(\mathbb{A}+\mathbb{P})_{-1})(J^{\mathbb{A},\mathbb{A}
+\mathbb{P}}\mathbb{B}+\bar{\mathbb{P}})\\
&\cdot(I-G_{\mathbb{A}+\mathbb{P},J^{\mathbb{A},\mathbb{A}
+\mathbb{P}}\mathbb{B}+\bar{\mathbb{P}},\mathbb{C},\mathbb{M}}(\lambda))^{-1}\mathbb{C}R(\lambda,\mathbb{A}+\mathbb{P}).     \end{align*}
It follows that from (\ref{phiab}) that
\begin{align*}
   \Phi_{\mathbb{A}+\mathbb{P},J^{\mathbb{A},\mathbb{A}
+\mathbb{P}}\mathbb{B}+\bar{\mathbb{P}}}=
\Phi_{\mathbb{A},I}(I-F_{\mathbb{A},I,\mathbb{P}})^{-1}
F_{\mathbb{A},\mathbb{B}+\bar{\mathbb{P}},\mathbb{P}}
+\Phi_{\mathbb{A},\mathbb{B}+\bar{\mathbb{P}}},
\end{align*}
which implies that
$R(\lambda,(\mathbb{A}+\mathbb{P})_{-1})(J^{\mathbb{A},\mathbb{A}
+\mathbb{P}}\mathbb{B}+\bar{\mathbb{P}})$ is the operator from $\left(
          \begin{array}{c}
            u \\
            v \\
          \end{array}
        \right)$ to $\left(
          \begin{array}{c}
            x \\
            w \\
          \end{array}
        \right)$ defined by the following algebraic equations
\begin{eqnarray*}
  \left\{
    \begin{array}{ll}
      \lambda\left(
                    \begin{array}{c}
                      x \\
                      w \\
                    \end{array}
                  \right)=
     \left(
                            \begin{array}{cc}
                              \mathfrak{A}_m & 0 \\
                              L & A_m \\
                            \end{array}
                          \right)
        \left(
                    \begin{array}{c}
                      x \\
                      w \\
                    \end{array}
                  \right)
      , & \hbox{ } t\geq 0\\
      \left(
                            \begin{array}{cc}
                              \mathfrak{G} & 0 \\
                              0 & P \\
                            \end{array}
                          \right)
      \left(
                    \begin{array}{c}
                      x \\
                      w \\
                    \end{array}
                  \right)
      =\left(
          \begin{array}{c}
            u \\
            v \\
          \end{array}
        \right).
    \end{array}
  \right.
\end{eqnarray*}
The first equation and third equation imply $x=R(\lambda,\mathfrak{A}_{-1})\mathfrak{B}u$.
Substitute it to the second equation to get
$(\lambda-A_m)w=LR(\lambda,\mathfrak{A}_{-1})\mathfrak{B}u$.
It is easy to obtain that the solutions of equations
\begin{align*}
    \left\{
      \begin{array}{ll}
        (\lambda-A_m)w_1=0, & \hbox{ } \\
        Pw_1=v, & \hbox{ }
      \end{array}
    \right.
\end{align*}
and
\begin{align*}
    \left\{
      \begin{array}{ll}
        (\lambda-A_m)w_2=LR(\lambda,\mathfrak{A}_{-1})\mathfrak{B}u, & \hbox{ } \\
        Pw_2=0, & \hbox{ }
      \end{array}
    \right.
\end{align*}
are $w_1=R(\lambda,A_{-1})Bv$
and $w_2=R(\lambda,A)LR(\lambda,\mathfrak{A}_{-1})\mathfrak{B}u$, respectively.
Hence $w=w_1+w_2=R(\lambda,A)LR(\lambda,\mathfrak{A}_{-1})\mathfrak{B}u
+R(\lambda,A_{-1})Bv.$
This implies
\begin{align*}
 R(\lambda,(\mathbb{A}+\mathbb{P})_{-1})(J^{\mathbb{A},\mathbb{A}
+\mathbb{P}}\mathbb{B}+\bar{\mathbb{P}})=&\left(
                         \begin{array}{cc}
                           R(\lambda,\mathfrak{A}_{-1})\mathfrak{B} & 0 \\
                           R(\lambda,A)LR(\lambda,\mathfrak{A}_{-1})\mathfrak{B} & R(\lambda,A_{-1})B \\
                         \end{array}
                       \right)\\
=&\left(
                         \begin{array}{cc}
                           e_\lambda & 0 \\
                           R(\lambda,A)Le_\lambda & R(\lambda,A_{-1})B \\
                         \end{array}
                       \right).
\end{align*}
By Lemma \ref{invert}, we compute the operator $(I-G_{\mathbb{A}+\mathbb{P},J^{\mathbb{A},\mathbb{A}
+\mathbb{P}}\mathbb{B}+\bar{\mathbb{P}},\mathbb{C},\mathbb{M}}(\lambda))^{-1}$
as follows
\begin{align*}
    &(I-G_{\mathbb{A}+\mathbb{P},J^{\mathbb{A},\mathbb{A}
+\mathbb{P}}\mathbb{B}+\bar{\mathbb{P}},\mathbb{C},\mathbb{M}}(\lambda))^{-1}
\\
=&\left(
    \begin{array}{cc}
      I-R(\lambda,A)L e_\lambda & -R(\lambda,A_{-1})B \\
      -MR(\lambda,A)L e_\lambda-Ke_\lambda & I-MR(\lambda,A_{-1})B\\
    \end{array}
  \right)^{-1}\\
=&\left(
    \begin{array}{cc}
      N_1 & [I-R(\lambda,A)L e_\lambda]^{-1}R(\lambda,A_{-1})BN_2 \\
      N_3 & N_2 \\
    \end{array}
  \right),
\end{align*}
where
$N_1=\{I-R(\lambda,A)Le_\lambda-R(\lambda,A_{-1})B(I-MR(\lambda,A_{-1})B)^{-1}[MR(\lambda,A)Le_\lambda+Ke_\lambda]\}^{-1},
$
$N_2=\{I-MR(\lambda,A_{-1})B-[MR(\lambda,A)Le_\lambda+Ke_\lambda](I-R(\lambda,A)Le_\lambda)^{-1}R(\lambda,A_{-1})B\}^{-1},
$ $N_3=[I-MR(\lambda,A_{-1})B]^{-1}[MR(\lambda,A)L
e_\lambda+Ke_\lambda]N_1$. The proof can be completed by simple
computation.
\end{proof}

\section{Linear Boundary Systems
with Delays in State and Boundary Output } In this section, we
consider boundary control systems with delays in state and boundary
output
\begin{eqnarray*}
(DLS)\left\{
  \begin{array}{ll}
        \dot{w}(t)=A_m w(t)+Lw_t, & \hbox{ }t\geq 0,\\
    Pw(t)=v(t),& \hbox{ }t\geq 0,\\
    y(t)=Mw(t)+Kw_t, & \hbox{ }t\geq 0,
  \end{array}
\right.
\end{eqnarray*}
where the operators $A_m$, $L$,
$P$, $M$ and $K$ are defined as in the above section;
for $t\geq 0$, $w_t$ is the history function defined by $w_t(\theta)=w(t+\theta),\ \theta\in[-r,0]$.

Let $\mathfrak{A}_m$ and $\mathfrak{B}_m$ be defined as in the above section. Then system $(DLS)$ can be converted to the following control system
\begin{align*}
\left\{
    \begin{array}{ll}
      \frac{d}{dt}\left(
                    \begin{array}{c}
                      x(t) \\
                      w(t) \\
                    \end{array}
                  \right)=
     \left(
                            \begin{array}{cc}
                              \mathfrak{A}_m & 0 \\
                              L & A_m \\
                            \end{array}
                          \right)
        \left(
                    \begin{array}{c}
                      x(t) \\
                      w(t) \\
                    \end{array}
                  \right)
      , & \hbox{ } t\geq 0\\
      \left(
                            \begin{array}{cc}
                              \mathfrak{G} & 0 \\
                              0 & P \\
                            \end{array}
                          \right)
      \left(
                    \begin{array}{c}
                      x(t) \\
                      w(t) \\
                    \end{array}
                  \right)
      =\left(
         \begin{array}{cc}
           0 & I \\
           0 & 0 \\
         \end{array}
       \right)
        \left(
                    \begin{array}{c}
                      x(t) \\
                      w(t) \\
                    \end{array}
                  \right)+\left(
                            \begin{array}{c}
                              0 \\
                              I \\
                            \end{array}
                          \right)v(t)
, & \hbox{ } t\geq 0\\
      y(t)=\left(
             \begin{array}{cc}
               K & M \\
             \end{array}
           \right)
        \left(
                    \begin{array}{c}
                      x(t) \\
                      w(t) \\
                    \end{array}
                  \right). & \hbox{ } t\geq 0.
    \end{array}
  \right.
\end{align*}

In order to prove the regularity of system ($DLS$), we have to introduce two lemmas.
\begin{lemma}\label{boundarycontrol1}
Assume that the boundary system $(L,G,Q)$ is a regular linear system
generated by $(\mathbb{A},\mathbb{B},\mathbb{C},\mathbb{D})$ on $(X,U,X)$ with
admissible feedback operator $I$. Then the system
\begin{align*}
  (OS)\left\{
   \begin{array}{ll}
   \dot{z}(t)&=Lz(t)\\
     Gz(t)&=Qz(t)+v(t)\\
   \end{array}
 \right.
\end{align*}
is an abstract linear control system generated by
$(\mathbb{A}^I,\mathbb{B}^I).$
\end{lemma}
\begin{remark}
In the special case that $\bar{Q}=0$,
Lemma \ref{boundarycontrol1}
has been proved in \cite[Theorem 10]{Hadd2008} and our
Lemma can be easily proved by the same procedure.
\end{remark}

The above lemma means that system $(OS)$ can obtained by taking the closed loop system of
\begin{align*}
    \left\{
      \begin{array}{ll}
        \dot{z}(t)=\mathbb{A}_{-1}z(t)+\mathbb{B}u(t), & \hbox{ } \\
        y(t)=\mathbb{C}_\Lambda^\mathbb{A}z(t)+\mathbb{D}u(t) & \hbox{ }
      \end{array}
    \right.
\end{align*}
under the feedback $u(t)=y(t)+v(t)$.

\begin{lemma}\label{cross}
Let $(A,B,P)$ generate a regular linear system with admissible feedback operator $I$ on $(X,U,U)$, and $(A,B,C)$ generate a regular linear system on $(X,U,Y)$. Then $(A^I,B^I,C_\Lambda^A)$ generates a regular linear system,
and there holds
$$F_{A^I,B^I,C_\Lambda^A}=F_{A,B,C}(I-F_{A,B,P})^{-1}.$$
where $A^I=(A+BP_\Lambda^A)|_X$ and $B^I=J^{A,A^I}B.$
\end{lemma}
\begin{proof}\ \
Consider the
operators $\tilde{B}:=(B, 0): X\times U\rightarrow X_{-1}$,
$\tilde{C}=\left(
                                                          \begin{array}{c}
                                                            P \\
                                                            C \\
                                                          \end{array}
                                                        \right)
: X\rightarrow X\times U$.

By the definition and Lemma \ref{bound}, it is easy to prove that
$(A,\tilde{B},\tilde{C})$ generates a regular linear system given by
\begin{eqnarray*}
  \Sigma_{A,\tilde{B},\tilde{C}}:=\left(
                                    \begin{array}{cc}
                                      T & (\Phi_{A,IB},0) \\
                                      \left(
                       \begin{array}{c}
                         \Psi_{A,P} \\
                         \Psi_{A,C} \\
                       \end{array}
                     \right)& \left(
                                 \begin{array}{cc}
                                   F_{A,B,P} & 0 \\
                                   F_{A,B,C} & 0 \\
                                 \end{array}
                               \right)
                      \\
                                    \end{array}
                                  \right).
\end{eqnarray*}
Observe that $I$ is an admissible feedback operator for $\Sigma_{A,B,P}$. So $$I_{X\times U}-\left(
                                       \begin{array}{cc}
                                         F_{A,B,P} & 0 \\
                                         F_{A,B,C} & 0 \\
                                       \end{array}
                                     \right)
=\left(
                                       \begin{array}{cc}
                                         I-F_{A,B,P} & 0 \\
                                         -F_{A,B,C} & I \\
                                       \end{array}
                                     \right)$$ is invert
and
$$\bigg(I_{X\times U}-\left(
                                       \begin{array}{cc}
                                         F_{A,B,P} & 0 \\
                                         F_{A,B,C} & 0 \\
                                       \end{array}
                                     \right)\bigg)^{-1}=\left(
                                       \begin{array}{cc}
                                         (I-F_{A,B,P})^{-1} & 0 \\
                                         F_{A,B,C}(I-F_{A,B,P})^{-1} & I \\
                                       \end{array}
                                     \right),$$
which indicates that $I_{X\times U}$ is an admissible feedback operator for $\Sigma_{A,\tilde{B},\tilde{C}}$. By theorem
\ref{relation}, it follows that $A^{I_{X\times U}}=(A_{-1}+\tilde{B}\tilde{C}_\Lambda^{\tilde{A}})|_X
=(A_{-1}+BP_\Lambda^A)|_X=A^I$,
$\tilde{B}^{I_{X\times U}}=J^{A,A^I}\tilde{B}=\left(
                                                \begin{array}{cc}
                                                  J^{A,A^I}B & 0 \\
                                                \end{array}
                                              \right)
=\left(
   \begin{array}{cc}
     B^I & 0 \\
   \end{array}
 \right)
$, $\tilde{C}^I_{X\times U}=\tilde{C}_\Lambda^{A}=
\left(
  \begin{array}{c}
    P_\Lambda^A \\
    C_\Lambda^A \\
  \end{array}
\right)
$
and
\begin{align*}
 F_{A^I,B^I,C_\Lambda^A}
&=\left(
    \begin{array}{cc}
      0 & I \\
    \end{array}
  \right)
F_{A^{I_{X\times U}},\tilde{B}^{I_{X\times U}},\tilde{C}^{I_{X\times U}}}\left(
      \begin{array}{c}
        I \\
        0 \\
      \end{array}
    \right)
\\
  &=F_{A,B,C}(I-F_{A,B,P})^{-1}.
\end{align*}
Observe that $C_\Lambda^A=\left(
                                 \begin{array}{cc}
                                   0 & I \\
                                 \end{array}
                               \right)
\left(
  \begin{array}{c}
    P_\Lambda^A \\
    C_\Lambda^A \\
  \end{array}
\right).$  Since $\bigg(A^I,B^I,\left(
  \begin{array}{c}
    P_\Lambda^A \\
    C_\Lambda^A \\
  \end{array}
\right)\bigg)$ is a regular linear system, $(A^I,B^I,    C_\Lambda^A)$ is also a regular linear system.
\end{proof}
\begin{remark}
In the special case that $B=I$, the above theorem says that
both $P$ and $C$ being admissible for $A$ implies that
$C$ is admissible for $A+P$. Such result has been proved by Hadd \cite{Hadd2006}.
This means that our result is a generalization of \cite{Hadd2006}.
\end{remark}

\begin{theorem}\label{fedin}
Assume that the boundary system $(L,G,Q)$ is a regular linear system
generated by $(\mathbb{A},\mathbb{B},\mathbb{P})$ on $(X,U,U)$ with
admissible feedback operator $I$ and $(L,G,K)$ is a regular linear
system on $(X,U,Y)$. Then the system
\begin{align*}
  \left\{
   \begin{array}{ll}
   \dot{z}(t)&=Lz(t)\\
     Gz(t)&=Qz(t)+v(t)\\
     y(t)&=Kz(t)
   \end{array}
 \right.
\end{align*}
is a regular linear system generated by
$(\mathbb{A}^I,\mathbb{B}^I,K,\bar{K}).$
\end{theorem}
\begin{proof}\ \
It follows from Lemma \ref{output} that
\begin{align*}
    Kz=K_\Lambda^Az+\bar{K}Gz,\ z\in Z.
\end{align*}
For any $z\in D(A^I)\subset D(L)$, $Gz=Qz.$
Since $(L,G,Q)$ is a regular linear system generated by $(\mathbb{A},\mathbb{B},\mathbb{P})$, we have the equality
$Kz=K_\Lambda^Az,\ z\in D(L).$
Then
\begin{align*}
    Kz=K_\Lambda^Az+\bar{K}Q_\Lambda^Az,\ z\in D(A^I).
\end{align*}
Observe that $(A^I,B^I,Q_\Lambda^A)$ is a regular linear system. By the Lemma \ref{cross}, $(A^I,B^I,K_\Lambda^A)$ generates a regular linear system. By virtue of \textbf{(S2)}, $(A^I,B^I,K)$ is a regular linea system. Combine Lemma \ref{output} and Lemma \ref{boundarycontrol1} to get that the boundary system is a regular linear system.
We compute the feedthrough operator
$$Dz=\lim_{\lambda\rightarrow +\infty}KR(\lambda,A^I_{-1})B^Iz=\lim_{\lambda\rightarrow +\infty}KR(\lambda,A_{-1})B[I-P_\Lambda^AR(\lambda,A_{-1})B]^{-1}z
=\bar{K}z.$$
This completes the proof.
\end{proof}
\begin{remark}
Let $K=Q$. Then the regular linear system is just the closed loop system of $(\mathbb{A},\mathbb{B},\mathbb{P})$ with
admissible feedback operator $I$.
\end{remark}

\begin{theorem}\label{uninput}
Assume that the boundary systems $(\mathfrak{A}_m,\mathfrak{G},K)$,
$(\mathfrak{A}_m,\mathfrak{G},L)$ and $(A_m,P,M)$ are regular linear
systems. Then system ($DLS$) is a regular linear system generated by
$$\bigg(\left(
\begin{array}{cc}
                              \mathfrak{A} & 0 \\
                              L & A \\
                            \end{array}
                          \right)^I,
J^{\left(
                            \begin{array}{cc}
                              \mathfrak{A} & 0 \\
                              L & A \\
                            \end{array}
                          \right),
\left(
                            \begin{array}{cc}
                              \mathfrak{A} & 0 \\
                              L & A \\
                            \end{array}
                          \right)^I}
J^{\left(
                            \begin{array}{cc}
                              \mathfrak{A} & 0 \\
                              0 & A \\
                            \end{array}
                          \right),
\left(
                            \begin{array}{cc}
                              \mathfrak{A} & 0 \\
                              L & A \\
                            \end{array}
                          \right)}\left(
                                    \begin{array}{c}
                                      0 \\
                                      B \\
                                    \end{array}
                                  \right)
,
                          \left(
                            \begin{array}{cc}
                              K & M \\
                            \end{array}
                          \right),\bar{M}
\bigg)$$ with admissible feedback operator $I$.
Moreover, the state
has the following mild expression
$$w(t)=T_A(t)w+\int_0^tT_A(t-s)[L_\Lambda w_s+\bar{L}w(s)]ds+\Phi_{A,B}(t)u.$$
\end{theorem}
\begin{proof}\ \
By Theorem \ref{feedback}, it follows that system
\begin{eqnarray*}
  (BCS)\left\{
    %
  \right.
\end{eqnarray*}
is a regular linear system with the generator
$$\bigg(\left(
                            \begin{array}{cc}
                              A & 0 \\
                              L & \mathfrak{A} \\
                            \end{array}
                          \right),
J^{\left(
                            \begin{array}{cc}
                               \mathfrak{A} & 0 \\
                              0 & A\\
                            \end{array}
                          \right),
\left(
                            \begin{array}{cc}
                              \mathfrak{A} & 0 \\
                              L & A \\
                            \end{array}
                          \right)}\left(
                            \begin{array}{cc}
                              \mathfrak{B} & 0 \\
                              0 & B \\
                            \end{array}
                          \right)+\left(
           \begin{array}{cc}
             0 & 0 \\
             \bar{L} & 0 \\
           \end{array}
         \right),
                          \left(
                            \begin{array}{cc}
                              0 & I \\
                              0 & 0 \\
                            \end{array}
                          \right)\bigg),$$
and $I$ is one of its admissible feedback operator.

For the closed loop system, the system operator $\left(
                            \begin{array}{cc}
                              \mathfrak{A} & 0 \\
                              L &  A\\
                            \end{array}
                          \right)^I$ is the restriction of operator
$$\left(
                            \begin{array}{cc}
                              \mathfrak{A} & 0 \\
                              L &  A\\
                            \end{array}
                          \right)_{-1}+
J^{\left(
                            \begin{array}{cc}
                              \mathfrak{A} & 0 \\
                              0 & A \\
                            \end{array}
                          \right),
\left(
                            \begin{array}{cc}
                              \mathfrak{A} & 0 \\
                              L & A \\
                            \end{array}
                          \right)}\left(
                            \begin{array}{cc}
                              0 & \mathfrak{B} \\
                              0 & 0 \\
                            \end{array}
                          \right)+\left(
                            \begin{array}{cc}
                              0 & 0 \\
                              0 & \bar{L} \\
                            \end{array}
                          \right)$$ on $X\times \partial X$ and control operator is $$J^{\left(
                            \begin{array}{cc}
                              \mathfrak{A} & 0 \\
                              L & A \\
                            \end{array}
                          \right),\left(
                            \begin{array}{cc}
                              \mathfrak{A} & 0 \\
                              L & A \\
                            \end{array}
                          \right)^I}J^{\left(
                            \begin{array}{cc}
                              \mathfrak{A} & 0 \\
                              0 & A \\
                            \end{array}
                          \right),
\left(
                            \begin{array}{cc}
                              \mathfrak{A} & 0 \\
                              L & A \\
                            \end{array}
                          \right)}\left(
                            \begin{array}{cc}
                              \mathfrak{B} & 0 \\
                              0 & B \\
                            \end{array}
                          \right)+\left(
                            \begin{array}{cc}
                              0 & 0 \\
                              \bar{L} & 0 \\
                            \end{array}
                          \right).$$

By Theorem \ref{exp}, it follows that
\begin{align*}
 w(t)=&T_A(t)w+\int_0^tT_A(t-s)[L_\Lambda x(s)+\bar{L}u_1(s)]ds+\Phi_{A,B}(t)u_2\\
=&T_A(t)w+\int_0^tT_A(t-s)[L_\Lambda w_s+\bar{L}w(s)]ds+\Phi_{A,B}(t)v.
\end{align*}
Here the feedback $u(t)=\left(
         \begin{array}{cc}
           0 & I \\
           0 & 0 \\
         \end{array}
       \right)
        \left(
                    \begin{array}{c}
                      x(t) \\
                      w(t) \\
                    \end{array}
                  \right)+\left(
                            \begin{array}{c}
                              0 \\
                              I \\
                            \end{array}
                          \right)v(t)$ is used.
Observe that $\left(
                \begin{array}{cc}
                  K & M \\
                \end{array}
              \right)=\left(
                        \begin{array}{cc}
                          0 & I \\
                        \end{array}
                      \right)
\left(
                        \begin{array}{cc}
                          0 & I \\
                          K & M \\
                        \end{array}
                      \right)
$. The combination of \textbf{(S2)} and Theorem \ref{exp} implies that
$$\bigg(\left(
                            \begin{array}{cc}
                              \mathfrak{A} & 0 \\
                              L & A \\
                            \end{array}
                          \right),
J^{\left(
                            \begin{array}{cc}
                              \mathfrak{A} & 0 \\
                              0 &  A\\
                            \end{array}
                          \right),
\left(
                            \begin{array}{cc}
                              \mathfrak{A} & 0 \\
                              L & A \\
                            \end{array}
                          \right)}\left(
                            \begin{array}{cc}
                              \mathfrak{B} & 0 \\
                              0 & B \\
                            \end{array}
                          \right)+\left(
           \begin{array}{cc}
             0 & 0 \\
             \bar{L} & 0 \\
           \end{array}
         \right),
                          \left(
                            \begin{array}{cc}
                              K & M \\
                            \end{array}
                          \right),
         \left(
                            \begin{array}{cc}
                              \bar{K} & \bar{M} \\
                            \end{array}
                          \right)
\bigg)$$ is a regular linear system. Observe that $\left(
                            \begin{array}{cc}
                              \bar{K} & \bar{M }\\
                            \end{array}
                          \right)\left(
                                   \begin{array}{c}
                                     0 \\
                                     I \\
                                   \end{array}
                                 \right)
                          =\bar{M}.$
Combine \textbf{(S2)} and Theorem \ref{fedin} to derive that
$(SLD)$ is a regular linear system generated by
\begin{align*}
    \bigg(&\left(
                            \begin{array}{cc}
                              \mathfrak{A} & 0 \\
                              L & A \\
                            \end{array}
                          \right)^I,
J^{\left(
                            \begin{array}{cc}
                              \mathfrak{A} & 0 \\
                              L &  A\\
                            \end{array}
                          \right),
\left(
                            \begin{array}{cc}
                              \mathfrak{A} & 0 \\
                              L &  A\\
                            \end{array}
                          \right)^I}J^{\left(
                            \begin{array}{cc}
                              \mathfrak{A} & 0 \\
                              0 & A \\
                            \end{array}
                          \right),
\left(
                            \begin{array}{cc}
                              \mathfrak{A} & 0 \\
                              L & A \\
                            \end{array}
                          \right)}\left(
                            \begin{array}{cc}
                              \mathfrak{B} & 0 \\
                              0 & B \\
                            \end{array}
                          \right)+\left(
           \begin{array}{cc}
             0 & 0 \\
             \bar{L} & 0 \\
           \end{array}
         \right),\\
         &\left(
                            \begin{array}{cc}
                              K & M \\
                            \end{array}
                          \right),
                          \bar{M}.
\bigg)
\end{align*} 
We compute
\begin{align*}
    &\Phi_{\left(
                            \begin{array}{cc}
                              \mathfrak{A} & 0 \\
                              L & A \\
                            \end{array}
                          \right)^I,
J^{\left(
                            \begin{array}{cc}
                              \mathfrak{A} & 0 \\
                              L & A \\
                            \end{array}
                          \right),
\left(
                            \begin{array}{cc}
                              \mathfrak{A} & 0 \\
                              L & A \\
                            \end{array}
                          \right)^I}J^{\left(
                            \begin{array}{cc}
                              \mathfrak{A} & 0 \\
                              0 & A \\
                            \end{array}
                          \right),
\left(
                            \begin{array}{cc}
                              \mathfrak{A} & 0 \\
                              L & A \\
                            \end{array}
                          \right)}\left(
                            \begin{array}{cc}
                             \mathfrak{B} & 0 \\
                              0 &  B \\
                            \end{array}
                          \right)+\left(
           \begin{array}{cc}
             0 & 0 \\
             \bar{L} & 0 \\
           \end{array}
         \right)}\\
=&\left(
                    \begin{array}{cc}
                      \Phi_{\mathfrak{A},\mathfrak{B}} & 0 \\
                      \Phi_{A,I}F_{\mathfrak{A},\mathfrak{B},L}+\Phi_{A,\bar{L}}& \Phi_{A,B} \\
                    \end{array}
                  \right).
\end{align*}
Then
\begin{align*}
    &R(\lambda,\left(
                            \begin{array}{cc}
                              \mathfrak{A}& 0 \\
                              L & A  \\
                            \end{array}
                          \right)^I_{-1})
\bigg(J^{\left(
                            \begin{array}{cc}
                              \mathfrak{A}& 0 \\
                              0 & A  \\
                            \end{array}
                          \right),
\left(
                            \begin{array}{cc}
                              \mathfrak{A} & 0 \\
                              L &  A\\
                            \end{array}
                          \right)}\left(
                            \begin{array}{cc}
                              \mathfrak{B} & 0 \\
                              0 & B \\
                            \end{array}
                          \right)+\left(
           \begin{array}{cc}
             0 & 0 \\
             \bar{L} & 0 \\
           \end{array}
         \right),
                          \left(
                            \begin{array}{cc}
                              K & M \\
                            \end{array}
                          \right)
\bigg)\\
=&\left(
                    \begin{array}{cc}
                      R(\lambda,\mathfrak{A}_{-1}B & 0 \\
                      R(\lambda,A)L_\Lambda R(\lambda,\mathfrak{A}_{-1})\mathfrak{B}+\bar{L}R(\lambda,A)& R(\lambda,A_{-1})B \\
                    \end{array}
                  \right).
\end{align*}
By \textbf{(S2)}, system $(DLS)$ is a regular linear system and the
transform function is given by
\begin{align*}
    & \left(
                                        \begin{array}{cc}
                                          K & M \\
                                        \end{array}
                                      \right)
\left(
                    \begin{array}{cc}
                      R(\lambda,\mathfrak{A}_{-1})B & 0 \\
                      R(\lambda,A)L_\Lambda R(\lambda,\mathfrak{A}_{-1})\mathfrak{B}+\bar{L}R(\lambda,A)& R(\lambda,A)B \\
                    \end{array}
                  \right)
\left(
  \begin{array}{c}
    0 \\
    I \\
  \end{array}
\right)\\
=&MR(\lambda,A_{-1})B.
\end{align*}
Since $(A,B,M,\bar{M})$ is a regular linear system, the feedthrough
operator is the limit $$\lim_{\lambda\rightarrow
+\infty}MR(\lambda,A_{-1})B=\bar{M}.$$ Hence system $(SLD)$ is a
regular linear system generated by
$$\bigg(\left(
                            \begin{array}{cc}
                              \mathfrak{A} & 0 \\
                              L & A \\
                            \end{array}
                          \right)^I,
J^{\left(
                            \begin{array}{cc}
                              \mathfrak{A} & 0 \\
                              L & A \\
                            \end{array}
                          \right),
\left(
                            \begin{array}{cc}
                              \mathfrak{A} & 0 \\
                              L & A \\
                            \end{array}
                          \right)^I}
J^{\left(
                            \begin{array}{cc}
                              \mathfrak{A} & 0 \\
                              0 & A \\
                            \end{array}
                          \right),
\left(
                            \begin{array}{cc}
                              \mathfrak{A} & 0 \\
                              L & A \\
                            \end{array}
                          \right)}\left(
                                    \begin{array}{c}
                                      0 \\
                                      B \\
                                    \end{array}
                                  \right)
,
                          \left(
                            \begin{array}{cc}
                              K & M \\
                            \end{array}
                          \right),\bar{M}
\bigg).$$

Moreover, $I$ is an admissible feedback operator of system $(SLD)$
because $I$ is admissible feedback operator of
$(A,B,M,\bar{M})$. The proof is therefore completed.
\end{proof}

\section{Linear Boundary Systems
with Delays in State, Input and Boundary Output } In this section,
we consider boundary control systems with delays in state and output
\begin{eqnarray*}
(DLS1)\left\{
  \begin{array}{ll}
        \dot{w}(t)=A_m w(t)+Lw_t+Eq_t, & \hbox{ }t\geq 0,\\
    Pw(t)=v(t),& \hbox{ }t\geq 0,\\
    y(t)=Mw(t)+Kw_t+Hq_t, & \hbox{ }t\geq 0,
  \end{array}
\right.
\end{eqnarray*}
and the boundary feedback systems
\begin{eqnarray*}
(BFS)\left\{
  \begin{array}{ll}
        \dot{w}(t)=A_m w(t)+Lw_t+Eq_t, & \hbox{ }t\geq 0,\\
    Pw(t)=Mw(t)+Kw_t+Hq_t, & \hbox{ }t\geq 0,
  \end{array}
\right.
\end{eqnarray*}
where the operators $\mathfrak{A}_m$, $L$, $P$, $M$ and $K$ are
defined as in the above section; for $t\geq 0$, $w_t$ is the history
function defined by $w_t(\theta)=w(t+\theta),\ \theta\in[-r,0]$.

Denote $\mathrm{X}=L^p([-r,0],U)$, $\mathrm{A}_m:=\frac{d}{d\theta}$ with
domain $D(\mathrm{A}_m):= W^{1,p}([-r,0],U)$, $\mathrm{G}W=W(0),
\forall \ W\in L^p([-r,0],U)$. Then the
boundary system
\begin{eqnarray*}
\left\{
  \begin{array}{ll}
    \dot{z}(t)=\mathrm{A}_mz(t),& \hbox{ } t\geq 0,\\
    \mathrm{G}z(t)=q(t)
  \end{array}
\right.
\end{eqnarray*}
is an abstract linear control system with the solutions $z=q_t$.
Let $(\mathrm{A},\mathrm{B})$ be the generator of such abstract linear control system. Then system $(DLS1)$ can be converted to the following control system
\begin{align*}
\left\{
    %
  \right.
\end{align*}
and $(BFS)$ can be converted to the following control system
\begin{align*}
\left\{
    %
                          \right)q(t)
, & \hbox{ } t\geq 0.
    \end{array}
  \right.
\end{align*}

\begin{theorem}\label{input}
Assume that the boundary systems $(\mathfrak{A}_m,\mathfrak{G},K)$,
$(\mathfrak{A}_m,\mathfrak{G},L)$, $(A,P,M)$,
$(\mathrm{A}_m,\mathrm{G},E)$ and $(\mathrm{A}_m,\mathrm{G},H)$ are
regular linear systems. Then system $(DLS1)$ is a regular linear
system. Moreover, the state has the following mild expression
$$w(t)=T_A(t)w+\int_0^tT_A(t-s)[L_\Lambda w_s+\bar{L}w(s)+Eq_s+\bar{E}q(s)]ds+\Phi_{A,B}(t)v.$$
\end{theorem}
\begin{proof}\ \
Denote $\mathbf{A}_m=\left(
                        \begin{array}{cc}
                          \mathfrak{A}_m & 0 \\
                          0 & \mathrm{A}_m \\
                        \end{array}
                      \right),
$ $\mathbf{G}=\left(
                 \begin{array}{cc}
                   \mathfrak{G} & 0 \\
                   0 & \mathrm{G} \\
                 \end{array}
               \right),$  $\mathbf{L}=\left(
                                        \begin{array}{cc}
                                          L & E \\
                                        \end{array}
                                      \right),$ $
\mathbf{K}=\left(
              \begin{array}{cc}
                K & H \\
              \end{array}
            \right)
.$ Denote by $\mathbf{A}$ the restriction of
$\mathbf{A}_m$ on $Ker\{\mathbf{G}\}$.
Then $\mathbf{A}=\left(
                        \begin{array}{cc}
                          \mathfrak{A} & 0 \\
                          0 & \mathrm{A} \\
                        \end{array}
                      \right)
$ and the control operator of boundary control system $(\mathbf{A}_m,\mathbf{G})$ is $\mathbf{B}=\left(
                        \begin{array}{cc}
                          \mathfrak{B} & 0 \\
                          0 & \mathrm{B} \\
                        \end{array}
                      \right)$.
We can easy to obtain that $(\mathbf{A}_m,\mathbf{G},\mathbf{K})$ and $(\mathbf{A}_m,\mathbf{G},\mathbf{L})$ are regular linear systems. Theorem 3.7 implies that
\begin{align*}
\left\{
    \begin{array}{ll}
      \frac{d}{dt}\left(
                    \begin{array}{c}
                      x(t) \\
                      z(t) \\
                      w(t) \\
                    \end{array}
                  \right)=\left(
                            \begin{array}{ccc}
                              \mathfrak{A}_m & 0 & 0 \\
                              0 & \mathrm{A}_m & 0\\
                              L & E & A_m \\
                            \end{array}
                          \right)
             \left(
                    \begin{array}{c}
                      x(t) \\
                      z(t) \\
                      w(t) \\
                    \end{array}
                  \right)
      , & \hbox{ } t\geq 0\\
      \left(
        \begin{array}{ccc}
          \mathfrak{G} & 0 & 0 \\
          0 & \mathrm{G} & 0 \\
          0 & 0 & P \\
        \end{array}
      \right)
      \left(
                    \begin{array}{c}
                      x(t) \\
                      z(t)\\
                      w(t) \\
                    \end{array}
                  \right)
      =u(t) & \hbox{ } t\geq 0.
    \end{array}
  \right.
\end{align*}
is an abstract linear control system with generator
$$\bigg(\left(
                            \begin{array}{cc}
                              \mathbf{A} & 0 \\
                              \mathbf{L} & A \\
                            \end{array}
                          \right),
J^{\left(
                            \begin{array}{cc}
                              \mathbf{A} & 0 \\
                              0 & A \\
                            \end{array}
                          \right),
\left(
                            \begin{array}{cc}
                              \mathbf{A} & 0 \\
                              \mathbf{L} & A \\
                            \end{array}
                          \right)}\left(
                            \begin{array}{cc}
                              \mathbf{B} & 0 \\
                              0 & B \\
                            \end{array}
                          \right)+\left(
           \begin{array}{cc}
             0 & 0 \\
             \bar{\mathbf{L}} & 0 \\
           \end{array}
         \right)\bigg).$$
Since operator
$\left(
         \begin{array}{ccc}
           0 & 0 & I \\
           0 & 0 & 0 \\
           0 & 0 & 0 \\
         \end{array}
       \right)$ is bounded, it follows by \textbf{(S5)} that
boundary system 
$$\bigg(\left(
                            \begin{array}{ccc}
                              \mathfrak{A}_m & 0 & 0 \\
                              0 & \mathrm{A}_m & 0\\
                              L & E & A_m \\
                            \end{array}
                          \right),
\left(
                            \begin{array}{ccc}
                              \mathfrak{G} & 0 & 0 \\
                              0 & \mathrm{G} & 0\\
                              0 & 0 & P \\
                            \end{array}
                          \right),
\left(
                            \begin{array}{ccc}
                              0 & 0 & I \\
                              0 & 0 & 0\\
                              0 & 0 & 0 \\
                            \end{array}
                          \right)\bigg)$$ 
is a regular linear system operated by 
$$\bigg(\left(
                            \begin{array}{cc}
                              \mathbf{A} & 0 \\
                              \mathbf{L} & A \\
                            \end{array}
                          \right),
J^{\left(
                            \begin{array}{cc}
                              \mathbf{A} & 0 \\
                              0 & A \\
                            \end{array}
                          \right),
\left(
                            \begin{array}{cc}
                              \mathbf{A} & 0 \\
                              \mathbf{L} & A \\
                            \end{array}
                          \right)}\left(
                            \begin{array}{cc}
                              \mathbf{B} & 0 \\
                              0 & B \\
                            \end{array}
                          \right)+\left(
           \begin{array}{cc}
             0 & 0 \\
             \bar{\mathbf{L}} & 0 \\
           \end{array}
         \right),\left(
         \begin{array}{ccc}
           0 & 0 & I \\
           0 & 0 & 0 \\
           0 & 0 & 0 \\
         \end{array}
       \right)\bigg).$$
and $I$ is its admissible feedback operator.
By the proof of Theorem \ref{uninput}, it follows that
$$\bigg(\left(
                            \begin{array}{cc}
                              \mathbf{A} & 0 \\
                              \mathbf{L} & A \\
                            \end{array}
                          \right),
J^{\left(
                            \begin{array}{cc}
                              \mathbf{A} & 0 \\
                              0 & A \\
                            \end{array}
                          \right),
\left(
                            \begin{array}{cc}
                              \mathbf{A} & 0 \\
                              \mathbf{L} & A \\
                            \end{array}
                          \right)}\left(
                            \begin{array}{cc}
                              \mathbf{B} & 0 \\
                              0 & B \\
                            \end{array}
                          \right)+\left(
           \begin{array}{cc}
             0 & 0 \\
             \bar{\mathbf{L}} & 0 \\
           \end{array}
         \right),\left(
                   \begin{array}{cc}
                     \mathbf{K} & M \\
                   \end{array}
                 \right)
\bigg)$$ is a regular linear system. Combine Lemma \ref{output} and
Theorem \ref{fedin} to get that
\begin{align*} \left\{
    \begin{array}{lll}
      \frac{d}{dt}\left(
                    \begin{array}{c}
                      x(t) \\
                      z(t) \\
                      w(t) \\
                    \end{array}
                  \right)=\left(
                            \begin{array}{ccc}
                              \mathfrak{A}_m & 0 & 0 \\
                              0 & \mathrm{A}_m & 0\\
                              L & E & A_m \\
                            \end{array}
                          \right)
             \left(
                    \begin{array}{c}
                      x(t) \\
                      z(t) \\
                      w(t) \\
                    \end{array}
                  \right)
      , & \hbox{ } t\geq 0\\
      \left(
        \begin{array}{ccc}
          \mathfrak{G} & 0 & 0 \\
          0 & \mathrm{G} & 0 \\
          0 & 0 & P \\
        \end{array}
      \right)
      \left(
                    \begin{array}{c}
                      x(t) \\
                      z(t)\\
                      w(t) \\
                    \end{array}
                  \right)
      =\left(
         \begin{array}{ccc}
           0 & 0 & I \\
           0 & 0 & 0\\
           0 & 0 & 0 \\
         \end{array}
       \right)
      \left(
                    \begin{array}{c}
                      x(t) \\
                      z(t)\\
                      w(t) \\
                    \end{array}
                  \right)+u(t) & \hbox{ } t\geq 0\\
      y(t)=\left(
             \begin{array}{ccc}
               K & H & M \\
             \end{array}
           \right)\left(
                    \begin{array}{c}
                      x(t) \\
                      z(t)\\
                      w(t) \\
                    \end{array}
                  \right)
       & \hbox{ } t\geq 0.
    \end{array}
  \right.
\end{align*}
is a regular linear system. By \textbf{(S2)}, the boundedness of
operator $\left(
                            \begin{array}{cc}
                              0 & 0 \\
                              I & 0 \\
                              0 & I \\
                            \end{array}
                          \right)$ implies that
system $(DLS1)$ is a regular linear system.
\end{proof}

\begin{theorem}\label{fe}
Assume that the boundary systems $(\mathfrak{A}_m,\mathfrak{G},K)$,
$(\mathfrak{A}_m,\mathfrak{G},L)$, $(\mathrm{A}_m,\mathrm{G},E)$ and
$(\mathrm{A}_m,\mathrm{G},H)$ are regular linear systems. Suppose
the boundary system $(A_m,P,M)$ to be a regularity linear system
with admissible feedback operator $I$. Then system $(BFS)$ is an
abstract linear control system.
\end{theorem}
\begin{proof}\ \
Let $\mathbf{A}_m,\ \mathbf{G},\ \mathbf{K}$ and $\mathbf{L}$ be defined as in Theorem \ref{input}.
Theorem \ref{input} has proved that $\bigg(\left(
                            \begin{array}{cc}
                              \mathbf{A}_m & 0 \\
                              \mathbf{L} & A \\
                            \end{array}
                          \right),\left(
                            \begin{array}{cc}
                              \mathbf{G} & 0 \\
                              0 & P \\
                            \end{array}
                          \right)\bigg)$ is an abstract linear control system. Denote $\mathbb{A}=\left(
                            \begin{array}{cc}
                              \mathbf{A} & 0 \\
                              0 & A \\
                            \end{array}
                          \right)$,
$\mathbb{B}=\left(
                            \begin{array}{cc}
                              \mathbf{B} & 0 \\
                              0 & B \\
                            \end{array}
                          \right)$,
$\mathbb{C}=\left(
                            \begin{array}{cc}
                              0 & \left(
                    \begin{array}{c}
                      I \\
                      0 \\
                    \end{array}
                  \right) \\
                              \mathbf{K} & M \\
                            \end{array}
                          \right)$,
$\mathbb{P}=\left(
                            \begin{array}{cc}
                              0 & 0 \\
                              \mathbf{L} & 0 \\
                            \end{array}
                          \right),$ $\mathbb{M}=\left(
                       \begin{array}{cc}
                         0 & 0 \\
                         \bar{\mathbf{K}} & \bar{M} \\
                       \end{array}
                     \right).$
By the proof of Theorem \ref{exp}, it follows that
boundary system
\begin{align}\label{dabian}
\left\{
    \begin{array}{ll}
      \frac{d}{dt}\left(
                    \begin{array}{c}
                      x(t) \\
                      z(t) \\
                      w(t) \\
                    \end{array}
                  \right)=\left(
                            \begin{array}{ccc}
                              \mathfrak{A}_m & 0 & 0 \\
                              0 & \mathrm{A}_m & 0\\
                              L & E & A_m \\
                            \end{array}
                          \right)
             \left(
                    \begin{array}{c}
                      x(t) \\
                      z(t) \\
                      w(t) \\
                    \end{array}
                  \right)
      , & \hbox{ } t\geq 0\\
      \left(
        \begin{array}{ccc}
          \mathfrak{G }& 0 & 0 \\
          0 & \mathrm{G} & 0 \\
          0 & 0 & P \\
        \end{array}
      \right)
      \left(
                    \begin{array}{c}
                      x(t) \\
                      z(t)\\
                      w(t) \\
                    \end{array}
                  \right)
      =u(t), & \hbox{ } t\geq 0,\\
y(t)=\left(
         \begin{array}{ccc}
           0 & 0 & I \\
           0 & 0 & 0 \\
           K & H & M \\
         \end{array}
       \right)
        \left(
                    \begin{array}{c}
                      x(t) \\
                      z(t)\\
                      w(t) \\
                    \end{array}
                  \right), & \hbox{ } t\geq 0.
    \end{array}
  \right.
\end{align}
is a regular linear system generated by $(\mathbb{A}+\mathbb{P},J^{\mathbb{A},\mathbb{A}
+\mathbb{P}}\mathbb{B}+\bar{\mathbb{P}},\mathbb{C},
\mathbb{M})$.
Moreover, the proof of Theorem \ref{feedback} implies that
$$F_{\mathbb{A}+\mathbb{P},J^{\mathbb{A},\mathbb{A}
+\mathbb{P}}\mathbb{B}+\bar{\mathbb{P}},\mathbb{C},
\mathbb{M}}=\left(
               \begin{array}{cc}
                 \left(
                   \begin{array}{c}
                     F_{A,I,I}F_{\mathbf{A},\mathbf{B},L}+F_{A,\bar{\mathbf{L}},I} \\
                     0 \\
                   \end{array}
                 \right)
 &
\left(
  \begin{array}{c}
    F_{A,B,I} \\
    0 \\
  \end{array}
\right)
 \\
                 F_{A,I,M}F_{\mathbf{A},\mathbf{B},\mathbf{L}}+
F_{\mathbf{A},\mathbf{B},\mathbf{K},\bar{\mathbf{K}}}
+F_{A,\bar{\mathbf{L}},M}& F_{A,B,M,\bar{M}} \\
               \end{array}
             \right).$$
Then we can obtain that $I$ is an admissible feedback operator of
boundary system $\ref{dabian}$ through the standard proof as in Theorem \ref{feedback}. By Lemma \ref{boundarycontrol1}, system $(BFS)$ is an abstract linear control system. This completes that proof.
\end{proof}

\section{Application to Population Dynamical Systems}

In this section, we have two tasks: the first one is to
study the well-posedness and asymptotic behavior of population dynamical system with
bounded delayed birth process
\begin{eqnarray}\label{population}
    \mbox{ }\left\{
                 \begin{array}{ll}
                   \frac{\partial w(t,a)}{\partial t}=-\frac{\partial w(t,a)}{\partial a}-\mu (a)w(t,a)-\alpha(a)w(t-r,a)\\
                   w(t,0)=\int_0^\infty\int_{-r}^0 \beta_1(\sigma,a)w(t+\sigma,a)d\sigma
                        da,t\geq 0 \\
                   w(s,a)=\phi(s,a),s\in[-r,0] \mbox{ and }a\geq 0,
                 \end{array}
               \right.
\end{eqnarray}
and unbounded delayed birth process
\begin{eqnarray}\label{population1}
    \mbox{ }\left\{
                 \begin{array}{ll}
                   \frac{\partial w(t,a)}{\partial t}=-\frac{\partial w(t,a)}{\partial a}-\mu (a)w(t,a)-\alpha(a)w(t-r,a)\\
                   w(t,0)=\int_0^\infty \beta_2(a)w(t-r,a)da, t\geq 0 \\
                   w(s,a)=\phi(s,a),s\in[-r,0] \mbox{ and }a\geq 0;
                 \end{array}
               \right.
\end{eqnarray}
the second one
is to prove that population equations
with death caused by harvesting (depended on delay)
\begin{eqnarray}\label{population2}
    \mbox{ }\left\{
                 \begin{array}{ll}
                   \frac{\partial w(t,a)}{\partial t}=-\frac{\partial w(t,a)}{\partial a}-\mu (a)w(t,a)-\alpha(a)w(t-r,a)-\eta (a)q(t-r,a)\\
                   w(t,0)=\int_0^\infty\int_{-r}^0 \beta_1(\sigma,a)w(t+\sigma,a)d\sigma
                        da,t\geq 0 \\
                   w(s,a)=\phi(s,a),s\in[-r,0] \mbox{ and }a\geq 0,
                 \end{array}
               \right.
\end{eqnarray}
and
\begin{eqnarray}\label{population3}
    \mbox{ }\left\{
                 \begin{array}{ll}
                   \frac{\partial w(t,a)}{\partial t}=-\frac{\partial w(t,a)}{\partial a}-\mu (a)w(t,a)-\alpha(a)w(t-r,a)-\eta (a)q(t-r,a)\\
                   w(t,0)=\int_0^\infty \beta_2(a)w(t-r,a)da, t\geq 0 \\
                   w(s,a)=\phi(s,a),s\in[-r,0] \mbox{ and }a\geq 0
                 \end{array}
               \right.
\end{eqnarray}
 are abstract linear control systems.
Here $w(t; a)$ represents the density of the population of age a at
time $t$, $\mu\in L^\infty_{loc}(R^+)$ and $\alpha \in L^\infty_{loc}(R^+)$ are respectively the death rate caused by natural death and pregnancy, $\beta_1\in
L^\infty([-r,0]\times R^+)$ and $\beta_2\in
L^\infty(R^+)$ are the birth rates, $\eta$ is the
death rate caused by harvesting. Denote
\begin{eqnarray*}
  \lim_{a\rightarrow \infty}\mu(a)=:\mu_\infty>0, \ \lim_{a\rightarrow \infty}\alpha(a)=:\alpha_\infty>0.
\end{eqnarray*}

We denote $X=U=V=L^1(R^+),$ $Y=R$. Then, systems (\ref{population}) and
(\ref{population1}) can be transformed to the form of system
(\ref{bdso}), and systems (\ref{population2}) and
(\ref{population3}) are transformed to the form of system
(\ref{bdsio}) with the operators:

$\bullet$ $A_m:=-\frac{d}{d\sigma}-\mu(\cdot)$ with domain
$D(A_m)=W^{1,1}(R^+)$;

$\bullet$ $(LF)(a)=-\alpha(a)F(-r,a), \forall F\in L^p([-r,0],L^1(R^+))$;

$\bullet$ $(Eh)(a)=-\eta(a)h(-r,a), \forall h\in L^p([-r,0],L^1(R^+))$;

$\bullet$ $M=0$;

$\bullet$ $Pf=f(0), \forall f\in L^1(R^+)$;

$\bullet$ $K_1g=\int_0^\infty \int_{-r}^0\beta_1(\sigma,a)g(\sigma,a)d\sigma da,\forall g\in
L^p([-r,0],L^1(R^+))$;

$\bullet$ $K_2g=\int_0^\infty \beta_2(a)g(-r,a)da,\forall g\in
L^p([-r,0],L^1(R^+))$;

$\bullet$ $H=0$.

Thus, $P\in L(W^{1,1}(R^+),\mathbb{C})$. The equations
\begin{eqnarray*}
 w(t; 0)=:B(t)=K_1w_t, t\geq 0,
\end{eqnarray*}
and
 \begin{eqnarray*}
 w(t; 0)=:B(t)=K_2w_t, t\geq 0,
\end{eqnarray*}
are the birth process, where $w_t:=w(t+\cdot)$ is the history
function. It has been shown in
 \cite[Proposition 2.1]{Greiner1984} that the spectrum
 $\sigma(A)$ is
 \begin{eqnarray*}
   \sigma(A)=\{\lambda\in
   \mathbb{C}:Re\lambda\leq-\mu_{\infty}\}.
 \end{eqnarray*}
Moreover, by \cite[(24)]{Greiner1984}, we have
\begin{eqnarray*}
  Ker(\lambda-A_m)=\left\{
                                \begin{array}{ll}
                                  <e^{-\int_0^\cdot(\lambda+\mu(s))ds}>, & \hbox{ }Re>-\mu_\infty, \\
                                  0, & \hbox{  } otherwise.
                                \end{array}
                              \right.
\end{eqnarray*}
It is not hard to obtain that $R(\lambda,A_{-1})B=e^{-\int_0^\cdot(\lambda+\mu(s))ds}.$
\begin{theorem}\label{regu}\cite{Mei2015}
The pair $(A,B)$ is an abstract linear control system. The triple
$(\mathfrak{A},\mathfrak{B},K_1)$ and
$(\mathfrak{A},\mathfrak{B},K_2)$ generate regular linear systems.
\end{theorem}

\begin{lemma}\label{reg}
The triple $(\mathfrak{A},\mathfrak{B},L)$ generates a regular
linear system.
\end{lemma}

With Theorem \ref{main}, Lemma \ref{reg} and Theorem \ref{regu}, we
obtain the following theorem.

\begin{theorem}
The population dynamical system (\ref{population}) is well-posed.
\end{theorem}

Observe that the operator $Ke_\lambda(I-R(\lambda,A)Le_\lambda)^{-1}R(\lambda,A_{-1})B$
has one-dimensional range, hence is compact. Thus, in Theorem \ref{yujie}, ``$\Longleftarrow$" can
be replaced by ``$\Longleftrightarrow$". On the other hand, observe
that $Re\lambda>-\mu_{\infty}$ implies $\lambda\in \rho(A)$.
Moreover, $Re\lambda>-\mu_{\infty}-\alpha_\infty$ implies $\lambda\in \rho(A+Le_\lambda).$
Therefore, we can obtain the following theorem.

\begin{theorem}\label{10}
Let $Re\lambda>-\mu_{\infty}$. Then
\begin{eqnarray*}
&&\lambda\in \sigma(\mathcal {A}_{L,K,0})\\
&\Longleftrightarrow& 1\in
\sigma\big(R(\lambda,A)Le_\lambda+R(\lambda,A_{-1})B
Ke_\lambda\big)\\
&\Longleftrightarrow& 1\in
\sigma\big(Ke_\lambda(I-R(\lambda,A)Le_\lambda)^{-1}R(\lambda,A_{-1})B\big).
\end{eqnarray*}
\end{theorem}

\begin{corollary}
Let $Re\lambda>-\mu_{\infty}$. Then
\begin{eqnarray*}
  \lambda\in \sigma(\mathcal {A}_{L,K,0})\Longleftrightarrow
  \lambda\in \sigma_P(\mathcal {A}_{L,K,0}).
\end{eqnarray*}
\end{corollary}
\begin{proof}\ \
The result is obtained directly from the combination of Theorem
\ref{8}, Theorem \ref{yujie} and Theorem \ref{10}.
\end{proof}

\begin{theorem}\label{eq}
Let $Re\lambda>-\mu_{\infty}$. Then,
$$\lambda\in \sigma(\mathcal {A}_{L,K_1,0})\Longleftrightarrow \xi_1(\lambda)=0$$
and
$$\lambda\in \sigma(\mathcal {A}_{L,K_2,0})\Longleftrightarrow \xi_2(\lambda)=0,$$
where
$$\xi_1(\lambda)=-1+\int_0^{+\infty}\int_{-r}^0\beta(\sigma,a)e^{\lambda \sigma}e^{-\int_0^a(\lambda+\mu(s)+e^{-\lambda r}\alpha(s))ds}d\sigma da$$
and $$\xi_2(\lambda)=-1+\int_0^\infty
\beta(a)e^{-\int_0^a(\lambda+\mu(s)+e^{-\lambda r}\alpha(s))ds}e^{-\lambda r}da$$
\end{theorem}
\begin{proof}\ \
By, it follows that $\lambda\in \sigma_P(\mathcal {A}_{L,K_i,0})$ if and only if
\begin{align*}
 1=K_ie_\lambda(I-R(\lambda,A)
Le_\lambda)^{-1}e^{-\int_0^\cdot(\lambda+\mu(s))ds},
\ i=1,2.
\end{align*}
In order to compute $(I-R(\lambda,A)
Le_\lambda)^{-1}e^{-\int_0^\cdot(\lambda+\mu(s))ds}$,
we solve the equation
\begin{align}\label{one}
   (I-R(\lambda,A)
Le_\lambda)f=e^{-\int_0^\cdot(\lambda+\mu(s))ds}.
\end{align}
Observe that $Le_\lambda f=-\alpha(\cdot)e^{-\lambda r}f(\cdot)$.
Find the solution of equations
\begin{align*}
   (\lambda-A)g=-\alpha(\cdot)e^{-\lambda r}f(\cdot)
\end{align*}
to get that $g=R(\lambda,A)
Le_\lambda f=-e^{-\int_0^\cdot(\lambda+\mu(s))ds}\int_0^\cdot
\alpha(s)e^{-\lambda r}e^{\lambda s+\int_0^s\mu(\sigma)d\sigma}f(s)ds.$
Then equation (\ref{one}) becomes
\begin{align*}
    f(a)+e^{-\int_0^a(\lambda+\mu(s))ds}\int_0^a
\alpha(s)e^{-\lambda r}e^{\lambda s+\int_0^s\mu(\sigma)d\sigma}f(s)ds
=e^{-\int_0^a(\lambda+\mu(s))ds}, \ a\geq 0,
\end{align*}
that is
\begin{align*}
   e^{\int_0^a(\lambda+\mu(s))ds}f(a)+\int_0^a
\alpha(s)e^{-\lambda r}e^{\lambda s+\int_0^s\mu(\sigma)d\sigma}f(s)ds
=1, \ a\geq 0.
\end{align*}
Let $m(a)=e^{\int_0^a(\lambda+\mu(s))ds}f(a),\ a\geq 0$.
Then the above equation convert to
\begin{align*}
   m(a)+\int_0^a
\alpha(s)e^{-\lambda r}m(s)ds
=1, \ a\geq 0,
\end{align*}
which is equivalent to the differential equation
\begin{align*}
    \left\{
      \begin{array}{ll}
        m'(a)+\alpha(a)e^{-\lambda r}m(a)=0, & \hbox{ } \\
        m(0)=1. & \hbox{ }
      \end{array}
    \right.
\end{align*}
This implies that $m(a)=e^{-e^{-\lambda r}\int_0^a\alpha(s)ds}, \ a\geq 0.$
$f(a)=e^{-\int_0^a(\lambda+\mu(s)+e^{-\lambda r}\alpha(s))ds},\ a\geq 0.$
Then
$\lambda\in \sigma_P(\mathcal {A}_{L,K_1,0})$ if and only if
\begin{align*}
 1=&K_1e_\lambda(I-R(\lambda,A)
Le_\lambda)^{-1}e^{-\int_0^\cdot(\lambda+\mu(s))ds}\\
=&K_1e_\lambda f\\
=&\int_0^{+\infty}\int_{-r}^0\beta_1(\sigma,a)e^{\lambda \sigma}e^{-\int_0^a(\lambda+\mu(s)+e^{-\lambda r}\alpha(s))ds}d\sigma da,
\end{align*}
and $\lambda\in \sigma_P(\mathcal {A}_{L,K_2,0})$ if and only if
\begin{align*}
 1=&K_2e_\lambda(I-R(\lambda,A)
Le_\lambda)^{-1}e^{-\int_0^\cdot(\lambda+\mu(s))ds}\\
=&K_2e_\lambda f\\
=&\int_0^{+\infty}\beta_2(a)e^{-\lambda r}e^{-\int_0^a(\lambda+\mu(s)+e^{-\lambda r}\alpha(s))ds}da.
\end{align*}
The proof is therefore completed.
\end{proof}

\begin{theorem}\label{ay}
The semigroups generated by $(\mathcal {A}_{L,K_1,0}, D(\mathcal {A}_{L,K_1,0}))$ and $(\mathcal {A}_{L,K_2,0}, D(\mathcal {A}_{L,K_2,0}))$ are positive and for $i=1,2$, the following
statements hold:

i)  $w_0(\mathcal {A}_{L,K_i,0})< 0\Leftrightarrow \xi_i(0)< 0$,

ii) $w_0(\mathcal {A}_{L,K_i,0})= 0\Leftrightarrow \xi_i(0)= 0$,

iii) $w_0(\mathcal {A}_{L,K_i,0})> 0\Leftrightarrow \xi_i(0)> 0$.
\end{theorem}

\begin{proof}\ \
The combination of $A$ being a generator of semigroup and $(\mathfrak{A},\mathfrak{B},L,\overline{L})$ generating a regular
linear systems implies that
$$\|R(\lambda,A)Le_\lambda\|\leq\|R(\lambda,A)\|\|Le_\lambda\|\rightarrow 0\ (\lambda\rightarrow +\infty).$$
Then $I-R(\lambda,A)Le_\lambda$ is invertible, $$[I-R(\lambda,A)Le_\lambda]^{-1}=\Sigma_{k=0}^\infty [R(\lambda,A)Le_\lambda]^k$$
and
$$\|(I-R(\lambda,A)Le_\lambda)^{-1}\|\leq\frac{1}{1-\|R(\lambda,A)\|\|Le_\lambda\|}$$
for enough big $\lambda$.
Since $(A,B)$ generates an abstract linear control system and $(\mathfrak{A},\mathfrak{B},K_i,\bar{K_i})$ generate regular
linear systems, we have that
$$\|R(\lambda,A)Le_\lambda+R(\lambda,A_{-1})BK_ie_\lambda\|\leq
\|R(\lambda,A)\|\|Le_\lambda\|+\|R(\lambda,A_{-1})B\|\|K_ie_\lambda\|\rightarrow 0$$
and
$$\|K_ie_\lambda[I-R(\lambda,A)Le_\lambda]^{-1}R(\lambda,A_{-1})B\|\leq
\|K_ie_\lambda\|\frac{1}{1-\|R(\lambda,A)\|\|Le_\lambda\|}R(\lambda,A_{-1})B\|\rightarrow 0$$
as $\lambda\rightarrow \infty$. Therefore, the operator
$N_1$ and $N_2$ are
invertible and there inverse are given by Neumann series.
By the proof of the above theorem, the equation
$$[I-R(\lambda,A)Le_\lambda]^{-1}f=g$$ is described by
\begin{align*}
   f(a)+e^{-\int_0^a(\lambda+\mu(s))ds}\int_0^a
\alpha(s)e^{-\lambda r}e^{\lambda s+\int_0^s\mu(\sigma)d\sigma}f(s)ds
=g(a), \ a\geq 0,
\end{align*}
which implies that operator $[I-R(\lambda,A)Le_\lambda]^{-1}$
is positive.
By \cite{Mei2015}, $R(\lambda,A_{-1})B$ is positive.
The positivities of $K_1e_\lambda$ and $K_2e_\lambda$
can be obtained by \cite{Piazzera2004} and \cite{Mei2015}, respectively.
Hence $N_2$ is positive.
Moreover, we can easy to obtain that
$N_1=[I-R(\lambda,A)Le_\lambda]^{-1}\{I+R(\lambda,A_{-1})BN_2Ke_\lambda[I-R(\lambda,A)Le_\lambda]^{-1}\}$
is also positive.
Then $W_1,\ W_2,\ W_5$ and $W_6$ are positive.
The positivities of $R(\lambda,\mathfrak{A})$ and $e_\lambda$
are obviously.
Therefore, operator
\begin{align*}
    &R(\lambda,\mathcal {A}_{L,K_i,0})\\
=&
\left(
  \begin{array}{cc}
    R(\lambda,\mathfrak{A})+e_\lambda W_1 & e_\lambda W_2 \\
     W_5& W_6 \\
  \end{array}
\right)
\end{align*}
is positive, which implies by \cite[Therem VI.1.15]{Engel2000} that
operator $\mathcal {A}_{L,K_i,0}$ generates a positive $C_0$-semigroup. Since the state space is an $AL$-space,
it follows from \cite[Theorem VI.1.15]{Engel2000} that $w_0(\mathcal
{A}_{L,K_i,0})=s(\mathcal {A}_{L,K_i,0})$ (spectrum boundness is qual to
growth boundness). Note that the function $\xi_i $ is continuous and
strictly decreasing with $\lim_{\lambda\rightarrow
-\infty}\xi_i(\lambda)=+\infty$ and $\lim_{\lambda\rightarrow
+\infty}\xi_i(\lambda)=1$. The rest of the proof is the same as the
proof of \cite[Theorem 13]{Piazzera2004}.
\end{proof}

The following result is directly obtained from Theorem \ref{eq} and
Theorem \ref{ay}.
\begin{corollary}
If $$\|\beta_i\|_\infty\int_0^\infty
e^{-\int_0^a\mu(s)ds}da<1,$$
then the growth bound of the semigroup generated by $\mathcal {A}_{L,K_i,0}$ satisfies $w_0(\mathcal {A}_{L,K_i,0})<0$. In
particular, all solutions (classical or mild) of (PE) are uniformly
exponentially stable.
\end{corollary}
\begin{remark}
We observe that the conditions of the above corollary are the same as Corollary 14 in \cite{Piazzera2004} for $i=1$ and Corollary 4.8 in \cite{Mei2015} for $i=2$. This means that if the conditions hold, the population systems are uniformly
exponentially stable both with and without death caused by pregnancies.
\end{remark}

From the above analysis, it is not hard to see that all the conditions of Theorem \ref{fe} are satisfied. Therefore the following theorem holds

\begin{theorem}
Population systems (\ref{population2}) and (\ref{population3}) are abstract linear control systems.
\end{theorem}


\end{document}